\pdfoutput=1
\documentclass[a4paper, 11pt, reqno]{amsart}
\calclayout
\usepackage{scrextend}
\usepackage{amssymb}
\usepackage{amsmath}
\usepackage{amsthm}
\usepackage[numbers]{natbib}
\usepackage{mathtools}
\usepackage{enumitem}
\usepackage{calc}
\usepackage{setspace}
\usepackage{graphicx}
\usepackage{tikz-cd}
\usepackage{mathscinet}
\usepackage{epigraph}
\usepackage{hyperref}
\usepackage{fancyhdr}
\usepackage[toc,page]{appendix}
\usepackage[top=3 cm, bottom=3 cm, left=3 cm, right=3cm]{geometry}
\makeatletter
\def\ps@plain{\ps@empty
}
\makeatother
\usepackage{lipsum}

\newcommand\blfootnote[1]{%
  \let\thempfn\relax
  \footnotetext[0]{{#1}}
}
\newtheorem{thm}{Theorem}[section]

\newtheorem{lem}[thm]{Lemma}

\theoremstyle{remark}
\newtheorem{ex}[thm]{Example}
\numberwithin{equation}{section}
\renewcommand\Re{\operatorname{Re}}
\renewcommand\Im{\operatorname{Im}}

\DeclareMathOperator{\id}{id}

\DeclareMathOperator{\Gal}{\mathrm{Gal}}

\DeclareMathOperator{\Nm}{\mathrm{Nm}}

\title{On the frequency of height values}
\author[G. A. Dill]{Gabriel A. Dill}
\address{University of Oxford, Mathematical Institute, Andrew Wiles Building, Radcliffe Observatory Quarter, Woodstock Road, Oxford OX2 6GG}
\email{gabriel.dill@maths.ox.ac.uk}
\date{\today}

\begin{document}

\begin{abstract}
We count algebraic numbers of fixed degree $d$ and fixed (absolute multiplicative Weil) height $\mathcal{H}$ with precisely $k$ conjugates that lie inside the open unit disk. We also count the number of values up to $\mathcal{H}$ that the height assumes on algebraic numbers of degree $d$ with precisely $k$ conjugates that lie inside the open unit disk. For both counts, we do not obtain an asymptotic, but only a rough order of growth, which arises from an asymptotic for the logarithm of the counting function; for the first count, even this rough order of growth exists only if $k \in \{0,d\}$ or $\gcd(k,d) = 1$. We therefore study the behaviour in the case where $0 < k < d$ and $\gcd(k,d) > 1$ in more detail. We also count integer polynomials of fixed degree and fixed Mahler measure with a fixed number of complex zeroes inside the open unit disk (counted with multiplicities) and study the dynamical behaviour of the height function.
\end{abstract}


\maketitle

\tableofcontents

\section{Introduction}

Let $\bar{\mathbb{Q}}$ denote the algebraic closure of $\mathbb{Q}$ in $\mathbb{C}$. For an algebraic number $\alpha \in \bar{\mathbb{Q}}$, let $H(\alpha)$ denote the (absolute multiplicative Weil) height of $\alpha$, as defined in Section 1.5 of \cite{MR2216774}. We have $H(\alpha) \in \bar{\mathbb{Q}} \cap [1,\infty)$ for all $\alpha \in \bar{\mathbb{Q}}$. By a well-known theorem of Northcott \cite{Northcott} (see also Theorem 1.6.8 in \cite{MR2216774}), there are at most finitely many algebraic numbers of bounded degree and bounded height. This article seeks to answer the question: ``How many $\alpha \in \bar{\mathbb{Q}}$ are there of fixed degree $d$ and fixed height $\mathcal{H}$?" In particular, we are interested in whether the height assumes many values, but each single value is assumed only rarely, or whether there are only few values that are however assumed very often. For fundamental properties of the height, we refer to Section 1.5 of \cite{MR2216774}.

Much is known about counting algebraic numbers or more generally points in $\mathbb{P}^n(\bar{\mathbb{Q}})$ of fixed degree (over $\mathbb{Q}$ or over any fixed number field) and bounded height: Schanuel first proved, in \cite{MR557080}, an asymptotic for the number of algebraic points of bounded height that are defined over a fixed number field. Further results, including the asymptotic for the number of quadratic points (over $\mathbb{Q}$) of bounded height, were obtained by Schmidt in \cite{MR1223249} and \cite{MR1330740}. If $n$ is larger than the degree of the point (over $\mathbb{Q}$), then Gao found and proved the correct asymptotic in \cite{MR2693933}. He also determined the correct order of magnitude for any $n$ and any degree (over $\mathbb{Q}$). Masser and Vaaler then counted algebraic numbers of fixed degree and bounded height in \cite{MR2487698} (over $\mathbb{Q}$) and \cite{MR2247898} (over any fixed number field). If the degree of the point (over any fixed number field) is at most slightly less than $\frac{2n}{5}$, then Widmer obtained the correct asymptotic in \cite{MR2558450}. More recently, Guignard \cite{Guignard} counted quadratic points (over any fixed number field) if $n \geq 3$; he also counted points whose degree (over any fixed number field) is an odd prime less than or equal to $n-2$. However, Guignard uses a slightly different height, corresponding to another choice of norm at the infinite places.

The same problem has also been studied for integral points, i.e. elements of $\bar{\mathbb{Q}}^n$ whose coordinates are algebraic integers: In Theorem 5.2 in Chapter 3 of \cite{Lang83}, Lang gives an asymptotic for the number of algebraic integers of bounded height that lie in a fixed number field (with an unspecified constant in the main term). The work \cite{ChernVaaler} of Chern and Vaaler, which was also used crucially in \cite{MR2487698}, yields an asymptotic for the number of algebraic integers of fixed degree over $\mathbb{Q}$ and bounded height. In \cite{MR3249885}, Barroero extended the results of Lang and Chern and Vaaler to count algebraic integers of fixed degree over any fixed number field and bounded height. Widmer counted, in \cite{MR3544624}, integral points of fixed degree (over any fixed number field) and bounded height under the assumption that the degree of the point is either $1$ or at most slightly less than $n$. In \cite{MR3687101}, Grizzard and Gunther counted (among other things) algebraic integers of fixed degree (over $\mathbb{Q}$), fixed norm, and bounded height. This last result is somewhat related to our work in that the $d$-th power of the height of an algebraic integer of degree $d$ (over $\mathbb{Q}$) with no conjugate inside the open unit disk is equal to the absolute value of its norm.

We emphasize that all these results give much more precise asymptotics than the ones obtained in this article. However, already when counting rational numbers of fixed height, Euler's phi function appears, so it is clear that such precise asymptotics cannot be obtained in general when counting algebraic numbers of fixed degree and fixed height. Instead, we strive to obtain an asymptotic for the logarithm of the associated counting function.

In order to state our results, we have to introduce some notation: The conjugates (over $\mathbb{Q}$) of an algebraic number are the complex zeroes of its minimal polynomial over $\mathbb{Q}$. While there is no nice asymptotic for the logarithm of the counting function associated to our question from the beginning, we have managed to obtain such an asymptotic in many cases if the number of conjugates that lie inside the open unit disk is also prescribed. For $d \in \mathbb{N} = \{1,2,\hdots\}$, $k \in \{0,\hdots,d\}$, and $\mathcal{H} \in [1,\infty)$, we set
\begin{multline}
A(k,d,\mathcal{H}) = \{ \alpha \in \mathbb{C}; [\mathbb{Q}(\alpha):\mathbb{Q}] = d, H(\alpha) = \mathcal{H}, \mbox{ and precisely }\nonumber\\
\mbox{ $k$ conjugates of $\alpha$ lie inside the open unit disk} \},\nonumber
\end{multline}
\[A(k,d) = \bigcup_{\mathcal{H} \geq 1}{A(k,d,\mathcal{H})},\]
\begin{multline}
B(k,d,\mathcal{H}) = \{ H(\alpha); \alpha \in \mathbb{C}, [\mathbb{Q}(\alpha):\mathbb{Q}] = d, H(\alpha) \leq \mathcal{H}, \mbox{ and precisely }\nonumber\\
\mbox{ $k$ conjugates of $\alpha$ lie inside the open unit disk} \},\nonumber
\end{multline}
and
\[ B(k,d) = \{ H(\alpha); \alpha \in A(k,d)\}. \]
By Northcott's theorem, the sets $A(k,d,\mathcal{H})$ and $B(k,d,\mathcal{H})$ are finite for all $\mathcal{H} \geq 1$. The main goal of this article is to measure the growth of $|A(k,d,\mathcal{H})|$ and $|B(k,d,\mathcal{H})|$ as functions of $\mathcal{H}$. As $A(k,d,\mathcal{H})$ is empty if $\mathcal{H} \not\in B(k,d)$, we consider $|A(k,d,\mathcal{H})|$ only for $\mathcal{H} \in B(k,d)$.

We set
\[ a(k,d) = \lim_{\stackrel{\mathcal{H} \in B(k,d)}{\mathcal{H} \to \infty}}{\frac{\log |A(k,d,\mathcal{H})|}{\log \mathcal{H}}} \]
and
\[ b(k,d) = \lim_{\mathcal{H} \to \infty}{\frac{\log |B(k,d,\mathcal{H})|}{\log \mathcal{H}}}\]
if these limits exist. It will follow from Lemma \ref{lem:growthlemma} that $A(k,d)$ is an infinite set. Thus, $B(k,d)$ contains arbitrarily large elements and at least the limes superior and inferior corresponding to $a(k,d)$ certainly exist.

We remark that it is not clear if the conjugates inside the open unit disk are the right thing to take into account here. The Galois group of the normal closure of $\mathbb{Q}(\alpha)$, the degree $[\mathbb{Q}\left(H(\alpha)^d\right):\mathbb{Q}]$, and the normal closure of $\mathbb{Q}\left(H(\alpha)^d\right)$ also seem to play an important role as will become apparent. Of course, these objects are not independent of one another (e.g. $k \in \{0,d\}$ is equivalent to $[\mathbb{Q}\left(H(\alpha)^d\right):\mathbb{Q}] = 1$).

The main results of this article can be summarized as follows:

\begin{thm}\label{thm:summary} Let $d \in \mathbb{N}$ and $k \in \{0,\hdots,d\}$. Then the following hold:

\begin{enumerate}[label=(\roman*)]
\item $b(0,d) = b(d,d) = d$ (Theorem \ref{thm:heightmain}(i)),
\item $a(0,d) = a(d,d) = d^2$ (Theorem \ref{thm:heightmain}(ii)),
\item $b(k,d) = d(d+1)$ if $0 < k < d$ (Theorem \ref{thm:heightmaintoo}(ii)),
\item $a(k,d) = 0$ if $0 < k < d$ and $\gcd(k,d) = 1$ (Theorem \ref{thm:unconditionalupperbound}), and
\item $a(k,d)$ does not exist if $0 < k < d$ and $\gcd(k,d) > 1$, but the corresponding limes superior and limes inferior are equal to $d(\gcd(k,d)-1)$ and $0$ respectively (Theorem \ref{thm:limsup} and Lemma \ref{lem:liminf}).
\end{enumerate}

\end{thm}

Demanding that the action of the Galois group of the normal closure of $\mathbb{Q}(\alpha)$ on the conjugates of $\alpha \in A(k,d,\mathcal{H})$ is sufficiently generic implies that there are few such $\alpha$ if $0 < k < d$. The following is our strongest result in this direction:

\begin{thm}[= Theorem \ref{thm:unconditionalupperbound}(iii)]
Let $d \in \mathbb{N}$, $\epsilon > 0$, and $k \in \{1,\hdots,d-1\}$. There exists a constant $C$, depending only on $d$, $k$, and $\epsilon$, such that for all $\mathcal{H} \in [1,\infty)$ we have
\begin{multline}
|\{\alpha \in A(k,d,\mathcal{H})\mbox{; the Galois group of the normal closure of $\mathbb{Q}(\alpha)$ acts} \nonumber\\
\mbox{primitively on the set of conjugates of $\alpha$}\}| \leq C\mathcal{H}^{\epsilon}.\nonumber
\end{multline}
\end{thm}

We also show that the height function together with the degree is in some sense ``almost injective'' if the degree is at least $2$:

\begin{thm}[= Theorem \ref{thm:coarse}]
Let $d \geq 2$. For every $\epsilon > 0$, there is $\mathcal{H}_0 = \mathcal{H}_0(d,\epsilon) \in \mathbb{R}$ such that
\[ \frac{|\{\alpha \in \mathbb{C};[\mathbb{Q}(\alpha):\mathbb{Q}] = d, H(\alpha) \leq \mathcal{H}\}|}{|\{H(\alpha);\alpha \in \mathbb{C}, [\mathbb{Q}(\alpha):\mathbb{Q}] = d, H(\alpha) \leq \mathcal{H}\}|} \leq \mathcal{H}^{\epsilon}\]
for all $\mathcal{H} \geq \mathcal{H}_0$.
\end{thm}

If $d = 4$ and $k = 2$, then we obtain finer results than those given by Theorem \ref{thm:summary} according to whether $[\mathbb{Q}(\mathcal{H}^4):\mathbb{Q}]$ equals $2$, $4$, or $6$:

\begin{thm}
Let $\epsilon > 0$. There exists a constant $C$, depending only on $\epsilon$, such that the following hold:
\begin{enumerate}[label=(\roman*)]
\item $|A(2,4,\mathcal{H})| \leq C\mathcal{H}^{\epsilon}$ for all $\mathcal{H} \in B(2,4)$ such that $[\mathbb{Q}(\mathcal{H}^4):\mathbb{Q}] = 6$ (Lemma \ref{lem:usefullemma}),
\item $|A(2,4,\mathcal{H})| \leq C\mathcal{H}^{\epsilon}$ for all $\mathcal{H} \in B(2,4)$ such that $[\mathbb{Q}(\mathcal{H}^4):\mathbb{Q}] = 4$ (Theorem \ref{thm:shishi}), and
\item for every $\kappa \in [0,4]$, there exists a sequence $(\mathcal{H}_n)_{n \in \mathbb{N}}$ in $B(2,4)$ such that $[\mathbb{Q}(\mathcal{H}_n^4):\mathbb{Q}] = 2$ for all $n \in \mathbb{N}$, $\lim_{n \to \infty}{\mathcal{H}_n} = \infty,$ and $\lim_{n \to \infty}{\frac{\log |A(2,4,\mathcal{H}_n)|}{\log \mathcal{H}_n}} = \kappa$ (Theorem \ref{thm:fail}).
\end{enumerate}
\end{thm}

In the construction in the proof of Theorem \ref{thm:fail}, the field $\mathbb{Q}(\mathcal{H}^4)$ is made to vary in an infinite set unless $\kappa = 4$. This suggests that in general fixing the field $\mathbb{Q}(\mathcal{H}^d)$ might lead to more uniform growth behaviour. The following is a simplified version of Theorem \ref{thm:conditionallowerbound}:

\begin{thm}\label{thm:conditionallowerboundsimplified}
Let $\delta \in (0,1)$ and $\epsilon > 0$ and let $K \subset \bar{\mathbb{Q}}$ be a fixed Galois extension of $\mathbb{Q}$. Let $k, d \in \mathbb{N}$ such that $0 < k < d$ and let $\mathcal{H} \in \bar{\mathbb{Q}} \cap [1,\infty)$ such that the normal closure of $\mathbb{Q}(\mathcal{H}^d)$ is equal to $K$.

Suppose that $\alpha \in A(k,d,\mathcal{H})$. Set $L = \mathbb{Q}(\alpha) \cap K$ and $l = d[L:\mathbb{Q}]^{-1}$ and let $\beta \in L$ be the $\mathbb{Q}(\alpha)/L$-norm of $\alpha$. There exists a constant $C = C(k,d,K,\delta,\epsilon) > 0$ such that if for every field embedding $\sigma: \mathbb{Q}(\beta) \hookrightarrow \mathbb{C}$, we have either $|\sigma(\beta)| \geq (1-\delta)^{-1}$ or $|\sigma(\beta)| \leq 1-\delta$, then
\[ |A(k,d,\mathcal{H})| \geq C\mathcal{H}^{d(l-1)-\epsilon}. \]
\end{thm}

Theorem \ref{thm:conditionallowerbound} is then used together with upper bounds for $|A(k,d,\mathcal{H})|$ from Theorem \ref{thm:unconditionalupperbound} for the determination of the limes superior corresponding to $a(k,d)$ in Theorem \ref{thm:limsup}. We also give examples that show the necessity of the dependence of $C$ on $K$ and $\delta$.

In Section \ref{sec:mahlermeasure}, we count polynomials with integer coefficients of fixed degree $d$ and fixed Mahler measure $\mathcal{M}$ as defined in Section 1.6.4 of \cite{MR2216774}. Among these polynomials, those that are irreducible in $\mathbb{Z}[t]$ are in a $1$-to-$d$ correspondence with the algebraic numbers of degree $d$ and height $\mathcal{M}^{\frac{1}{d}}$. However, we also count the polynomials that are reducible in $\mathbb{Z}[t]$ and this leads to somewhat simpler results although even fewer of the considered limits exist. We obtain Theorem \ref{thm:mahlermeasure}, an analogue of Theorem \ref{thm:summary} in this context.

Following a suggestion of Norbert A'Campo, we study the dynamical behaviour of the height function in Section \ref{sec:heightdynamic}. The dynamical behaviour of the Mahler measure has been studied initially by Dubickas in \cite{MR1904083} and \cite{MR2037988} and subsequently by Zhang in \cite{Zhang} as well as by Fili, Pottmeyer, and Zhang in \cite{FPZ19}. We obtain the following result:

\begin{thm}[= Theorem \ref{thm:attractors}]
Let $\alpha \in \bar{\mathbb{Q}}$ and define inductively $\alpha_0 = \alpha$, $\alpha_n = H(\alpha_{n-1})$ ($n \in \mathbb{N}$). Then either there exist $N, a \in \mathbb{N}$ and $b \in \mathbb{Q}$, $b > 0$, such that $\alpha_n = a^b$ for all $n \geq N$ or $\lim_{n \to \infty}{\alpha_n} = 1$.
\end{thm}

In particular, the periodic points of $H$ are precisely the $a^b$ for $a \in \mathbb{N}$ and $b \in \mathbb{Q}$, $b > 0$.

Our proofs are mostly elementary. Our constructions of many algebraic numbers of a given height rely on point counting results for lattices by Barroero-Widmer in \cite{MR3264671} (in the proofs of Theorem \ref{thm:heightmain} and Lemma \ref{lem:growthlemma}) and Widmer in \cite{Widmer} (in the proof of Theorem \ref{thm:conditionallowerbound}). The first of these results generalizes a theorem of Davenport in \cite{MR0043821} while the second one generalizes a theorem of Skriganov in \cite{Skriganov}.

The main result of \cite{MR3264671} is formulated in an arbitrary o-minimal structure; we will however apply it only in the structure of semialgebraic sets, where a subset of $\mathbb{R}^n$ ($n \in \mathbb{N}$) is called semialgebraic or definable (in the structure of semialgebraic sets) if it is a finite union of sets defined by a finite number of polynomial equations and inequalities with real coefficients. By the Seidenberg-Tarski theorem, the structure of semialgebraic sets is o-minimal, which implies that besides polynomial equations and inequalities with real coefficients, we can also use existential and universal quantifiers to define semialgebraic sets. For a general introduction to o-minimal structures, see \cite{MR1633348}.

For a real number $\xi$, we denote by $[\xi]$ the largest integer which does not exceed $\xi$. We use $\phi$ to denote Euler's phi function and $\mu$ to denote the M\"obius function. For a finite field extension $L/K$, we denote the corresponding field norm by $N_{L/K}$. If $K$ is a number field, then we denote its ring of integers by $\mathcal{O}_K$. The norm of an ideal $\mathcal{I}$ of $\mathcal{O}_K$ is denoted by $N(\mathcal{I})$. The imaginary unit in $\mathbb{C}$ is denoted by $\sqrt{-1}$ and the real and imaginary part of a complex number are denoted by $\Re$ and $\Im$ respectively.

For a real-valued function $f$ on $S \subset \mathbb{R}^n$, we write $\mathcal{O}(f)$ for any function $g: S \to \mathbb{R}$ such that there exists a constant $C = C(f,g) \geq 0$ with $|g(s)| \leq Cf(s)$ for all $s \in S$. If $n = 1$, $S$ is unbounded, and $f(s) > 0$ for $|s|$ large enough, we say that a function $g: S \to \mathbb{R}$ is of growth order $o(f)$ if $\lim_{s \in S,|s| \to \infty}{\frac{|g(s)|}{f(s)}} = 0$.

If $\alpha$ is an algebraic number of degree $d$, a minimal polynomial of $\alpha$ in $\mathbb{Z}[t]$ is an irreducible element of $\mathbb{Z}[t]$ that has $\alpha$ as a zero. There are two choices for a minimal polynomial of $\alpha$ in $\mathbb{Z}[t]$ as $(\mathbb{Z}[t])^{\ast} = \{\pm1\}$. The following simple observation will be used at different places throughout this article: If $a$ is the leading coefficient of a minimal polynomial of $\alpha$ in $\mathbb{Z}[t]$ and $\alpha_1,\hdots,\alpha_{d-k}$ are the conjugates of $\alpha$ that lie outside the open unit disk, then $H(\alpha)^d = |a||\alpha_1|\cdots|\alpha_{d-k}| = \pm a\alpha_1\cdots\alpha_{d-k}$ (see Propositions 1.6.5 and 1.6.6 in \cite{MR2216774}). We can write $\pm \alpha_i$ instead of $|\alpha_i|$ ($i=1,\hdots,d-k$) since the non-real conjugates appear in complex conjugate pairs and the real conjugates are equal to their absolute value up to sign. 

\section{The case $k \in \{0,d\}$}\label{sec:two}

In this section, we treat the case where $k \in \{0,d\}$, which is the easiest one to resolve.

\begin{thm}\label{thm:heightmain}
Let $d \in \mathbb{N}$. The following hold:
\begin{enumerate}[label=(\roman*)]
\item $b(0,d) = b(d,d) = d$, and
\item $a(0,d) = a(d,d) = d^2$.
\end{enumerate}
(In particular, all these limits exist.)
\end{thm}

\begin{proof}
(i) Eisenstein's criterion shows that all real positive $d$-th roots of integers between $2$ and $\mathcal{H}^d$ that are congruent to $2$ modulo $4$ belong to $B(0,d,\mathcal{H})$. Using that the height of a non-zero algebraic number is equal to the height of its inverse, we deduce that they also belong to $B(d,d,\mathcal{H})$. Also, every element of $B(0,d,\mathcal{H})$ or $B(d,d,\mathcal{H})$ is a real positive $d$-th root of some integer between $1$ and $\mathcal{H}^d$. So $\mathcal{H}^d \geq |B(0,d,\mathcal{H})| \geq \frac{1}{5}\mathcal{H}^d$ for $\mathcal{H}$ large enough and the same holds for $|B(d,d,\mathcal{H})|$. It follows that (i) holds.

(ii) Let us define 
\begin{multline}\label{eq:definableset}
Z = \Big\{(w_0,\hdots,w_{d-1},T) \in \mathbb{R}^d \times \mathbb{R}; w_0 > 0, \exists x_1,\hdots,x_d,\mbox{ }y_1,\hdots,y_d \in \mathbb{R}:\\
x_j^2+y_j^2 \geq 1\mbox{ }\forall j = 1,\hdots,d,\mbox{ }g_j(x_1,y_1,\hdots,x_d,y_d) = 0\mbox{ }\forall j =0,\hdots,d-1,\\
\mbox{and }f_j(x_1,y_1,\hdots,x_d,y_d,T) = w_j\forall j=0,\hdots,d-1\Big\},
\end{multline}
where $f_j(x_1,y_1,\hdots,x_d,y_d,T) =$
\[\Re\left(\frac{(-1)^{d-j} T}{(x_1+\sqrt{-1}y_1)\cdots(x_d+\sqrt{-1}y_d)}\sigma_j(x_1+\sqrt{-1}y_1,\hdots,x_d+\sqrt{-1}y_d)\right) ,\]
and $g_j(x_1,y_1,\hdots,x_d,y_d) =$
\[\Im\left(\frac{1}{(x_1+\sqrt{-1}y_1)\cdots(x_d+\sqrt{-1}y_d)}\sigma_j(x_1+\sqrt{-1}y_1,\hdots,x_d+\sqrt{-1}y_d)\right)\]
for $j = 0,\hdots, d-1$, where $\sigma_j$ is the $j$-th elementary symmetric polynomial in $d$ variables.

The set $Z$ is definable in the o-minimal structure of all semialgebraic subsets of $\mathbb{R}^n$ ($n \in \mathbb{N}$). Let $\pi: \mathbb{R}^{d} \times \mathbb{R} \to \mathbb{R}^d$ be the canonical projection. For $T \in \mathbb{R}$, $T \neq 0$, the set $Z_T = \pi(Z \cap (\mathbb{R}^{d} \times \{T\}))$ parametrizes polynomials $w_0t^d + \cdots + w_{d-1}t + T$ of degree $d$ with real coefficients and positive leading coefficient that have no complex zeroes inside the open unit disk and whose constant coefficient is equal to $T$. Note that $Z_T = |T| \cdot Z_{T/|T|}$ ($T \neq 0$) and that the coordinates of a point in $Z_T$ can all be bounded by some constant multiple of $|T|$, depending on $d$. It follows that the volume of $Z_T$ is $|T|^d$ times the volume of $Z_{T/|T|}$ ($T \neq 0$) and that the volume of any orthogonal projection of $Z_T$ on some $j$-dimensional coordinate subspace of $\mathbb{R}^{d}$ has $j$-dimensional volume at most a constant multiple of $|T|^{j}$, depending on $d$.

We then deduce from Theorem 1.3 in \cite{MR3264671} that
\[ \left|\left|Z_T \cap \mathbb{Z}^d\right| - V_{T/|T|}|T|^d\right| = \mathcal{O}(|T|^{d-1})\]
for $|T| \geq 1$, where $V_{u}$ is the volume of $Z_{u}$ for $u \in \{\pm 1\}$. Here and in the rest of this proof, the implicit constants in the $\mathcal{O}$ notation depend only on $d$. The volume $V_{u}$ is positive for $u \in \{\pm 1\}$ since
\[ |w_0z^d + \cdots + w_{d-1}z + u| > 1 - \frac{d}{2d} > 0\]
for all $(w_0,\hdots,w_{d-1}) \in [1/(4d),1/(2d)] \times [-1/(2d),1/(2d)]^{d-1}$ and all $z \in \mathbb{C}$ with $|z| < 1$ and therefore $[1/(4d),1/(2d)] \times [-1/(2d),1/(2d)]^{d-1} \subset Z_u$. We have
\begin{multline}
N_d(T) := |\{P(t) = at^d+\cdots\pm T \in \mathbb{Z}[t]; a > 0\mbox{, all complex zeroes of $P$ are at least}\nonumber\\
\mbox{$1$ in absolute value}\}| = \left|Z_T \cap \mathbb{Z}^d\right| +\left|Z_{(-T)} \cap \mathbb{Z}^d\right|
= (V_1+V_{-1})T^d + \mathcal{O}(T^{d-1})\nonumber\end{multline}
for $T \in \mathbb{N}$.

If we define
\begin{multline}
\tilde{N_d}(T) = |\{P(t) = at^d+\cdots\pm T \in \mathbb{Z}[t]; a > 0, \gcd(a,\hdots,\pm T) = 1,\nonumber\\
\mbox{all complex zeroes of $P$ are at least $1$ in absolute value}\}|,\nonumber
\end{multline}
then we have $N_d(T) = \sum_{S|T}{\tilde{N_d}(S)}$. Using M\"obius inversion together with an elementary bound for the divisor function, we deduce that
\[ \tilde{N_d}(T) = \sum_{S|T}{\mu(S)N_d\left(\frac{T}{S}\right)} = (V_1+V_{-1})T^d\left(\sum_{S|T}{\frac{\mu(S)}{S^d}}\right)+\mathcal{O}\left(T^{d-\frac{1}{2}}\right).\]
Here $\sum_{S|T}{\frac{\mu(S)}{S^d}} = \prod_{p|T}{\left(1-\frac{1}{p^d}\right)}$ is at most $1$ and at least $\frac{\phi(T)}{T}$. In fact, for $d \geq 2$, the product is at least $\prod_{k=2}^{\infty}{\left(1-\frac{1}{k^2}\right)} = \frac{1}{2}$, so bounded from below uniformly.

What we really want is
\begin{multline}
\hat{N_d}(T) = |\{P(t) = at^d+\cdots\pm T \in \mathbb{Z}[t]; a > 0, \gcd(a,\hdots,\pm T) = 1, \mbox{$P$ is irreducible in $\mathbb{Q}[t]$,}\nonumber\\
\mbox{and all complex zeroes of $P$ are at least $1$ in absolute value}\}|\nonumber
\end{multline}
since $|A(0,d,\mathcal{H})| = d\hat{N_d}(\mathcal{H}^d)$ if $\mathcal{H}^d \in \mathbb{N}$, but the contribution of the reducible polynomials to $\tilde{N}_d(T)$ is at most
\[ \sum_{e=1}^{\left[\frac{d}{2}\right]}\sum_{R|T}{\tilde{N}_e(R)\tilde{N}_{d-e}\left(\frac{T}{R}\right)} = \sum_{e=1}^{\left[\frac{d}{2}\right]}\sum_{R|T}{\mathcal{O}(R^{2e-d}T^{d-e})} = \mathcal{O}\left(T^{d-\frac{1}{2}}\right).\]

Hence, we obtain that
\[ d(V_1+V_{-1})\frac{\phi(\mathcal{H}^d)}{\mathcal{H}^d}\mathcal{H}^{d^2} - \mathcal{O}\left(\mathcal{H}^{d\left(d-\frac{1}{2}\right)}\right) \leq |A(0,d,\mathcal{H})| \leq d(V_1+V_{-1})\mathcal{H}^{d^2} + \mathcal{O}\left(\mathcal{H}^{d\left(d-\frac{1}{2}\right)}\right)\]
for $\mathcal{H}^d \in \mathbb{N}$. Since $\mathcal{H}^d \in \mathbb{N}$ for all $\mathcal{H} \in B(0,d)$, we deduce (ii) from elementary lower bounds for $\phi(\mathcal{H}^d)$, at least for $a(0,d)$. For $a(d,d)$ we can repeat the same argument, but counting $\frac{1}{\alpha}$ instead of $\alpha$ and replacing $x_j^2+y_j^2 \geq 1$ by $x_j^2+y_j^2 > 1$ in \eqref{eq:definableset}.
\end{proof}

One can say even more about the sets $B(0,d)$ and $B(d,d)$. We denote by $\mathbb{N}^{\frac{1}{d}}$ the set of the positive real $d$-th roots of all natural numbers.

\begin{lem}\label{lem:finitecomplement}
Let $d \in \mathbb{N}$. We have
\[ B(0,d) = \left\{ \begin{array}{lll}
\mathbb{N}^{\frac{1}{d}}\backslash\{1\} & \quad & \mbox{if $d \not\in \phi(\mathbb{N})$,}\\
\mathbb{N}^{\frac{1}{d}}  & \quad & \mbox{if $d \in \phi(\mathbb{N})$,}
\end{array}\right.\]
and
\[ B(d,d) = \left\{\begin{array}{lll}
\mathbb{N}^{\frac{1}{d}}\backslash\{1\}  & \quad & \mbox{if $d > 1$,}\\
\mathbb{N}^{\frac{1}{d}}  & \quad & \mbox{if $d = 1$}.
\end{array}\right.\]
\end{lem}

\begin{proof} (As suggested by G. R\'emond.)
It follows from Propositions 1.6.5 and 1.6.6 in \cite{MR2216774} that $B(0,d)$ and $B(d,d)$ are both contained in $\mathbb{N}^{\frac{1}{d}}$. In the case $d = 1$, the lemma follows from $H(n) = H(n^{-1}) = n$ for all $n \in \mathbb{N}$ together with $H(0) = 1$, so we assume that $d \geq 2$.

If $[\mathbb{Q}(\alpha):\mathbb{Q}] = d \geq 2$ and $H(\alpha) = 1$ for some $\alpha \in \bar{\mathbb{Q}}$, then $\alpha$ is a root of unity by Kronecker's theorem (Theorem 1.5.9 in \cite{MR2216774}), so $\alpha \in A(0,d,1)$ and $d = \phi(n)$ for some $n \in \mathbb{N}$. On the other hand, if $d = \phi(n)$ for some $n \in \mathbb{N}$, then any primitive $n$-th root of unity belongs to $A(0,d,1)$. It follows that $1$ never belongs to $B(d,d)$ if $d > 1$ and that $1$ belongs to $B(0,d)$ if and only if $d \in \phi(\mathbb{N})$.

Let now $N$ be a natural number that is greater than or equal to $2$. We want to show that the positive real $d$-th root $N^{\frac{1}{d}}$ of $N$ belongs to $B(0,d) \cap B(d,d)$. For this, we define a natural number $p$ as follows: If $N = 2$, we set $p = 1$. If $N \geq 3$, then we let $p \in \mathbb{N}$ be a prime number such that $p < N$ and $p$ does not divide $N$. Such a prime number always exists: If $N = 3$, we set $p = 2$. If $N \geq 4$ and no such prime number existed, then $N$ would be divisible by the product $\Pi$ of all prime numbers that are smaller than $N$. Now $\Pi-1 \geq 5$ must have a prime factor and this prime factor must be greater than or equal to $N$. It follows that $\Pi \leq N \leq \Pi-1$, a contradiction.

The polynomial $Nt^d-p$ is irreducible in $\mathbb{Z}[t]$ by the coprimality of $p$ and $N$ together with Eisenstein's criterion (applied to $pt^d-N$ if $N = 2$). The complex zeroes of this polynomial belong to $A\left(d,d,N^{\frac{1}{d}}\right)$ and their inverses belong to $A\left(0,d,N^{\frac{1}{d}}\right)$. It follows that $N^{\frac{1}{d}} \in B(0,d) \cap B(d,d)$. This completes the proof of the lemma.
\end{proof}

\section{Some useful lemmas}\label{sec:three}

In this section, we collect some simple but useful lemmas. The first one shows that specifying the number of conjugates inside the open unit disk does not change the growth rate obtained by Masser and Vaaler in \cite{MR2487698}.

\begin{lem}\label{lem:growthlemma}
Let $d \in \mathbb{N}$ and $k \in \{0,\hdots,d\}$. The limit
\[\lim_{\mathcal{H}\to\infty}{\frac{\sum_{\mathcal{H}' \leq \mathcal{H}}{|A(k,d,\mathcal{H}')|}}{\mathcal{H}^{d(d+1)}}}\]
exists and is positive.
\end{lem}

Lemma \ref{lem:growthlemma} implies that $A(k,d)$ is infinite and so $B(k,d)$ contains arbitrarily large elements.

\begin{proof}
We first remark that
\begin{multline}
\sum_{\mathcal{H}' \leq \mathcal{H}}{|A(k,d,\mathcal{H}')|} = |\{\alpha \in \mathbb{C}; [\mathbb{Q}(\alpha):\mathbb{Q}] = d, H(\alpha) \leq \mathcal{H}, \mbox{ and}\nonumber\\
\mbox{precisely $k$ conjugates of $\alpha$ lie inside the open unit disk}\}|.\nonumber
\end{multline}

We can again apply Theorem 1.3 from \cite{MR3264671} to the following definable family of semialgebraic sets:
\begin{multline}\label{eq:anotherdefinableset}
\tilde{Z} = \Bigg\{(w_0,\hdots,w_{d},T) \in \mathbb{R}^{d+1} \times \mathbb{R}; T \geq 1, w_0 > 0, \exists x_1,\hdots,x_d, \\
y_1,\hdots,y_d \in \mathbb{R}: x_j^2+y_j^2 < 1\mbox{ }\forall j = 1,\hdots,k,\mbox{ }x_j^2+y_j^2 \geq 1\mbox{ }\forall j = k+1,\hdots,d, \\
g_j(x_1,y_1,\hdots,x_d,y_d) = 0\mbox{ }\forall j =1,\hdots,d,\\
w_0f_j(x_1,y_1,\hdots,x_d,y_d) = w_j\forall j=1,\hdots,d, w_0^2\prod_{j=k+1}^{d}(x_{j}^2+y_{j}^2) \leq T^2\Bigg\},
\end{multline}
where
\[f_j(x_1,y_1,\hdots,x_d,y_d) = (-1)^j\Re\left(\sigma_j(x_1+\sqrt{-1}y_1,\hdots,x_d+\sqrt{-1}y_d)\right)\]
and
\[ g_j(x_1,y_1,\hdots,x_d,y_d) = \Im\left(\sigma_j(x_1+\sqrt{-1}y_1,\hdots,x_d+\sqrt{-1}y_d)\right)\]
for $j = 1,\hdots, d$ and the $\sigma_j$ are again the elementary symmetric polynomials in $d$ variables. If again $\tilde{Z}_T = \pi(\tilde{Z} \cap (\mathbb{R}^{d+1} \times \{T\}))$ for the projection $\pi: \mathbb{R}^{d+1} \times \mathbb{R} \to \mathbb{R}^{d+1}$ and $T \geq 1$, then it is easy to see that all coordinates of a point in $\tilde{Z}_T$ are bounded by some constant multiple of $T$, depending on $d$, that $\tilde{Z}_T = T \cdot \tilde{Z_1}$, and that $\tilde{Z}_1$ has non-empty interior.

Similarly as above, $N_{d,k}(T) := |\tilde{Z}_T \cap \mathbb{Z}^{d+1}|$ counts the number of polynomials $P(t) \in \mathbb{Z}[t]$ of degree $d$ with positive leading coefficient and precisely $k$ complex zeroes inside the open unit disk (counted with multiplicities) such that the product of the leading coefficient and the absolute values of the complex zeroes outside the open unit disk, each absolute value raised to the power of the respective zero's multiplicity, is at most $T$. If $\tilde{N}_{d,k}(T)$ denotes the number of such polynomials with coprime coefficients, then we have that $N_{d,k}(T) = \sum_{n=1}^{\infty}{\tilde{N}_{d,k}\left(\frac{T}{n}\right)}$.

Using another M\"{o}bius inversion and Theorem 1.3 from \cite{MR3264671}, we deduce that
\[ \tilde{N}_{d,k}(T) = C\left(\sum_{n=1}^{\infty}{\frac{\mu(n)}{n^{d+1}}}\right)T^{d+1} + \mathcal{O}(T^d\log \max\{2,T\})\]
for some constant $C > 0$, where $C$ as well as the implicit constant in the $\mathcal{O}$ notation depend only on $d$ and $k$. The proof of Lemma 2 in \cite{MR2487698} shows that the number of reducible polynomials that we count in this way is of lower growth order. We can therefore deduce the lemma by setting $T = \mathcal{H}^d$.
\end{proof}

The next lemma follows straightforwardly from Lemma \ref{lem:growthlemma}.

\begin{lem}\label{lem:anotherusefullemma}
Let $d \in \mathbb{N}$ and $k \in \{0,\hdots,d\}$. If the limits $a(k,d)$ and $b(k,d)$ both exist, then $a(k,d) + b(k,d) = d(d+1)$.
\end{lem}

\begin{proof}
If they added up to some smaller number, we would immediately obtain a contradiction with Lemma \ref{lem:growthlemma} for $\mathcal{H}$ big enough, so suppose they add up to some bigger number. If $b(k,d) = 0$, then $a(k,d) > d(d+1)$ and we immediately get a contradiction with Lemma \ref{lem:growthlemma} for $\mathcal{H}$ big enough. So we can assume that $b(k,d) > 0$.

We can find some $\epsilon \in (0,1)$ such that $(1-\epsilon)b(k,d) + (1-\epsilon)a(k,d) > d(d+1)$ and then we can find $\delta \in (0,\epsilon)$ such that $(1-\delta)b(k,d) > (1+\delta)(1-\epsilon)b(k,d)$ and $(1-\epsilon)b(k,d) + (1-\delta)(1-\epsilon)a(k,d) > d(d+1)$. For $\mathcal{H} \geq 1$ large enough, it follows from the definitions of $a(k,d)$ and $b(k,d)$ that

\begin{multline}
\sum_{\mathcal{H}' \in B(k,d,\mathcal{H})}{|A(k,d,\mathcal{H}')|} \geq \sum_{\stackrel{\mathcal{H}' \in B(k,d,\mathcal{H})}{\mathcal{H}' \geq \mathcal{H}^{1-\epsilon}}}{|A(k,d,\mathcal{H}')|} \nonumber\\
\geq \left(\mathcal{H}^{(1-\delta)b(k,d)}-\mathcal{H}^{(1+\delta)(1-\epsilon)b(k,d)}\right)\mathcal{H}^{(1-\delta)(1-\epsilon)a(k,d)}.\nonumber
\end{multline}

As $(1-\delta)b(k,d) > (1+\delta)(1-\epsilon)b(k,d)$ and $\delta < \epsilon$, the right-hand side grows asymptotically faster than
\[\mathcal{H}^{(1-\epsilon)b(k,d) + (1-\delta)(1-\epsilon)a(k,d)}.\]
Since $(1-\epsilon)b(k,d)+(1-\delta)(1-\epsilon)a(k,d) > d(d+1)$, this again contradicts Lemma \ref{lem:growthlemma}.
\end{proof}

The next lemma is the first and weakest in a series of results saying that for $k \in \{1,\hdots,d-1\}$, there cannot exist too many $\alpha \in A(k,d,\mathcal{H})$ whose Galois group is ``large".

\begin{lem}\label{lem:usefullemma}
Let $d \in \mathbb{N}$, $\epsilon > 0$, and $k \in \{1,\hdots,d-1\}$. There exists a constant $C = C(k,d,\epsilon)$ such that
\begin{multline}
|\{\alpha \in A(d-k,d,\mathcal{H})\mbox{; the Galois group of the normal closure of $\mathbb{Q}(\alpha)$ acts} \nonumber\\
\mbox{transitively on the $k$-element subsets of the set of conjugates of $\alpha$}\}| \leq C\mathcal{H}^{\epsilon}\nonumber
\end{multline}
for all $\mathcal{H} \geq 1$.

Furthermore, we have
\[ |A(k,d,\mathcal{H})| \leq C\mathcal{H}^{\epsilon}\]
for all $\mathcal{H} \geq 1$ with $[\mathbb{Q}(\mathcal{H}^d):\mathbb{Q}] = \binom{d}{k}$.
\end{lem}

\begin{proof}
Let $\alpha \in A(d-k,d,\mathcal{H})$ and assume either that the Galois group of the normal closure of $\mathbb{Q}(\alpha)$ acts transitively on the $k$-element subsets of the set of conjugates of $\alpha$ or that $[\mathbb{Q}(\mathcal{H}^d):\mathbb{Q}] = \binom{d}{d-k}$. Now, for such an $\alpha$ we have $\mathcal{H}^d = H(\alpha)^d = \pm a \alpha_1 \cdots \alpha_k$, where $a > 0$ is the leading coefficient of a minimal polynomial of $\alpha$ in $\mathbb{Z}[t]$ and $\alpha_1, \hdots, \alpha_k$ are the conjugates of $\alpha$ that do not lie inside the open unit disk. By assumption, we have $0 < k < d$. We can assume without loss of generality that $\alpha = \alpha_1$ since $\alpha_1$ determines $\alpha$ up to $d$ possibilities.

Now note that
\[ a\alpha^k = a \alpha_1^k = \frac{\left(a\prod_{j=1}^{k}\alpha_j\right)\prod_{i=2}^{k}\left(a\alpha_{k+1}\left(\prod_{\stackrel{j=1}{j \neq i}}^{k}\alpha_j\right)\right)}{\left(a\prod_{j=2}^{k+1}{\alpha_j}\right)^{k-1}},\]
where $\alpha_{k+1}$ is a conjugate of $\alpha$, distinct from the $\alpha_j$ ($j=1,\hdots,k$) (here we use that $k < d$). The numerator and denominator of the right-hand side are products of conjugates of $\pm \mathcal{H}^d$ by our assumption on either the Galois group of the normal closure of $\mathbb{Q}(\alpha)$ or the degree of $\mathcal{H}^d$. So $a\alpha^k$ is determined by $\mathcal{H}$ up to finitely many possibilities (bounded in terms of only $d$ and $k$), so it can be assumed fixed. The same holds for $a\alpha_j^k$ for all $j = 1,\hdots, d$ by conjugating. And $a\alpha^k$ together with $a$ determines $\alpha$ up to $k$ possibilities (here we need that $k > 0$), so it remains to bound the number of possibilities for $a$.

But $a^{d-k}|b|^{k} = \prod_{j=1}^{d}{a|\alpha_j|^k}$ is already determined up to finitely many possibilities (bounded independently of $\mathcal{H}$), where $b$ is the constant coefficient of a minimal polynomial of $\alpha$ in $\mathbb{Z}[t]$, and $a$ has to divide this natural number as $k < d$. Since $|\prod_{j=1}^{d}{a\alpha_j^k}| \leq \mathcal{H}^{d^2}$, it follows from well-known bounds for the divisor function that there are at most $C'(d,\epsilon)\mathcal{H}^{\epsilon}$ possibilities for $a$.
\end{proof}

The next two lemmas contain general facts from algebraic number theory that will be useful at several places in this article.

\begin{lem}\label{lem:yetmoreusefullemmata}
Suppose that $\alpha \in \bar{\mathbb{Q}}$ with $[\mathbb{Q}(\alpha):\mathbb{Q}] = d$. Let $\alpha_1, \hdots, \alpha_d$ be the conjugates of $\alpha$ and let $a \in \mathbb{Z}$ be the leading coefficient of a minimal polynomial of $\alpha$ in $\mathbb{Z}[t]$. Let $S$ be a subset of $\{1,\hdots,d\}$. Then $a\prod_{s \in S}{\alpha_s}$ is an algebraic integer.
\end{lem}

\begin{proof}
Let $v$ be a finite place of $\mathbb{Q}(\alpha_1,\hdots,\alpha_d)$ and $|\cdot|_v$ an associated absolute value. We have
\[ \left|a\prod_{s \in S}{\alpha_s}\right|_v \leq |a|_v\prod_{i=1}^{d}{\max\{1,|\alpha_i|_v\}}.\]
From the Gauss lemma (Lemma 1.6.3 in \cite{MR2216774}) and the definition of $a$, we deduce that
\[|a|_v\prod_{i=1}^{d}{\max\{1,|\alpha_i|_v\}} = 1.\]
As $v$ was arbitrary, the lemma follows.
\end{proof}

\begin{lem}\label{lem:givennormboundedheight}
Let $K \subset \bar{\mathbb{Q}}$ be a number field, $N \in \mathbb{Z}$, $\mathcal{H} \geq 1$, and $\epsilon > 0$. Set $D = [K:\mathbb{Q}]$. There exists a constant $C = C(D,\epsilon)$ such that
\[ |\{\alpha \in \mathcal{O}_K; N_{K/\mathbb{Q}}(\alpha) = N, H(\alpha) \leq \mathcal{H}\}| \leq C\mathcal{H}^{\epsilon}.\]
\end{lem}

Lemma \ref{lem:givennormboundedheight} essentially follows from the proof of Proposition 2.5 in \cite{MR2006555}. For the reader's convenience, we reproduce the proof here.

\begin{proof}
We can assume without loss of generality that $0 < |N| \leq \mathcal{H}^D$ since otherwise the set whose cardinality we wish to bound has at most one element.

Let $U_K$ denote the group of algebraic units in $K$. We call two elements of $K\backslash\{0\}$ associate if their quotient belongs to $U_K$. It follows from \cite{BorSha}, pp. 219--220, our bound for $|N|$ in terms of $\mathcal{H}$, and elementary bounds for the divisor function that the number of pairwise non-associate elements of $\mathcal{O}_K\backslash\{0\}$ with $K/\mathbb{Q}$-norm $N$ is bounded by $C'(D,\epsilon)\mathcal{H}^{\frac{\epsilon}{2}}$.

Hence we can assume that $\alpha = \alpha_0\xi$ for some $\xi \in U_K$ and fixed $\alpha_0 \in \mathcal{O}_K$ with $N_{K/\mathbb{Q}}(\alpha) = N_{K/\mathbb{Q}}(\alpha_0) = N$ and $\max\{H(\alpha), H(\alpha_0)\} \leq \mathcal{H}$. It follows that $H(\xi) \leq \mathcal{H}^2$. We want to bound the number of possibilities for $\xi$.

Let $\sigma_i: K \hookrightarrow \mathbb{C}$ denote the distinct embeddings of $K$ in $\mathbb{C}$ ($i=1,\hdots,D$) and set
\[ \nu(\eta) = (\log |\sigma_1(\eta)|,\hdots,\log |\sigma_D(\eta)|)\]
for $\eta \in U_K$. Then $\nu$ is a group homomorphism from the multiplicative group $U_K$ to the additive group $\mathbb{R}^D$ and its image $\nu(U_K) \subset \mathbb{R}^D$ is a discrete free $\mathbb{Z}$-module.

Since $H(\xi) \leq \mathcal{H}^2$, we have that $\nu(\xi)$ belongs to the cube $[-2D\log \mathcal{H},2D\log \mathcal{H}]^{D}$. This cube can be covered by at most $C''(D,\epsilon)\mathcal{H}^{\frac{\epsilon}{2}}$ translates of the unit cube $[0,1]^D$. Since $\nu$ is a group homomorphism, it therefore suffices to show that the number of $\eta \in U_K$ with $\nu(\eta) \in [-1,1]^D$ is bounded by a constant depending only on $D$. For $\eta \in U_K$ define
\[ P_{\eta}(t) = \prod_{i=1}^{D}{(t-\sigma_i(\eta))} \in \mathbb{Z}[t]\]
so that $P_{\eta}(\eta) = 0$. If $\nu(\eta) \in [-1,1]^D$, then the absolute value of each coefficient of $P_{\eta}$ is bounded by $\exp(D(1+\log 2))$. This completes the proof of the lemma.
\end{proof}

\section{The case $k \in \{1,d-1\}$ or $d$ prime}

In this section, we completely resolve the cases where $k \in \{1,d-1\}$ or $d$ is prime. We also determine $b(k,d)$ for all $k$ and $d$. Although many of the results in this section will be superseded by Theorem \ref{thm:unconditionalupperbound}, we have included them because they can be proved in a different, somewhat easier way.

\begin{thm}\label{thm:heightmaintoo}
Let $d \in \mathbb{N}$ and $k \in \{0,\hdots,d\}$. Then the following hold:
\begin{enumerate}[label=(\roman*)]
\item $a(1,d) = a(d-1,d) = 0$ if $d \geq 2$,
\item $b(k,d) = d(d+1)$ if $d \geq 2$ and $0 < k < d$, and
\item $a(k,d) = 0$ if $d$ is prime and $0 < k < d$.
\end{enumerate}
(In particular, all these limits exist.)
\end{thm}

Theorem \ref{thm:heightmaintoo}(ii) implies together with Lemma \ref{lem:anotherusefullemma} that for $d \geq 2$ and $0 < k < d$, $a(k,d)$ must be equal to $0$ if it exists.

\begin{proof}
(i) This follows from Lemma \ref{lem:usefullemma} as the Galois group of the normal closure of $\mathbb{Q}(\alpha)$ always acts transitively on the $1$-element and the $(d-1)$-element subsets of the set of conjugates of $\alpha$.

(ii) It follows from Lemma \ref{lem:growthlemma} that $|B(k,d,\mathcal{H})| = \mathcal{O}\left(\mathcal{H}^{d(d+1)}\right)$ for $\mathcal{H} \geq 1$. If the equality in (ii) is false or the limit $b(k,d)$ does not exist, it follows that there is some $\epsilon > 0$ such that there exist arbitrarily large $\mathcal{H} \geq 1$ such that $|B(k,d,\mathcal{H})| \leq \mathcal{H}^{d(d+1)-\epsilon}$. Lemma \ref{lem:usefullemma} implies that for $\mathcal{H}' \in [1,\infty)$, the number of $\alpha \in A(k,d,\mathcal{H}')$ with Galois group isomorphic to the full symmetric group $S_d$ is bounded from above by $C\mathcal{H}'^{\frac{\epsilon}{2}}$ for some constant $C = C(k,d,\epsilon)$. Furthermore, the number of $\alpha \in \bigcup_{\mathcal{H}' \leq \mathcal{H}}{A(k,d,\mathcal{H}')}$ with Galois group not isomorphic to the full symmetric group is of growth order $o\left(\mathcal{H}^{d(d+1)}\right)$ (see \cite{MR1512878}). But by Lemma \ref{lem:growthlemma}, the number of $\alpha$ of degree $d$ with precisely $k$ conjugates inside the open unit disk and height at most $\mathcal{H}$ grows asymptotically like some constant positive multiple of $\mathcal{H}^{d(d+1)}$, which yields a contradiction for $\mathcal{H}$ large enough.

(iii) We follow a similar strategy as in the proof of Lemma \ref{lem:usefullemma}. Let $d$ be a prime number, $0 < k < d$, $\mathcal{H} \in [1,\infty)$, and $\alpha \in A(k,d,\mathcal{H})$. The Galois group of the normal closure of $\mathbb{Q}(\alpha)$ must contain an element of order $d$ since $d$ is prime and the Galois group acts transitively on the $d$-element set of conjugates of $\alpha$. Since $d$ is prime, such an element of order $d$ must act as a $d$-cycle on the conjugates of $\alpha$. If these conjugates are $\alpha_1,\hdots,\alpha_d$, we can assume without loss of generality that this $d$-cycle acts on them by acting on the indices as $(12\cdots(d-1)d)$. We have $\mathcal{H}^d = H(\alpha)^d = \pm a\prod_{i \in I}{\alpha_i}$ for some $I \subset \{1,\hdots,d\}$ with $|I| = d-k$ and $a \in \mathbb{N}$ the leading coefficient of a minimal polynomial of $\alpha$ in $\mathbb{Z}[t]$. We aim to write some $l$-th power of $a\alpha_1^{d-k}$ as a quotient of products of conjugates of $\pm \mathcal{H}^d$, where $l \in \mathbb{N}$ and the number of conjugates that appear are bounded in terms of $k$ and $d$ only. Once this is achieved, we can conclude as in the proof of Lemma \ref{lem:usefullemma}.

To a (formal) product $\prod_{i=1}^{d}{\alpha_i^{e_i}}$ with $e_i \in \mathbb{Z}$ we associate a vector $(e_1,\hdots,e_d) \in \mathbb{Z}^d$. Let $v \in \mathbb{Z}^d$ be the vector associated to $\prod_{i \in I}{\alpha_i}$. Consider the $\mathbb{Z}$-module $\Lambda$ generated by $A^i v$ ($i=0,\hdots,d-1$), where $A$ is a permutation matrix corresponding to the cycle $(12\cdots d)$. If finite (which we will later prove it to be), the index $[\mathbb{Z}^d : \Lambda]$ can be bounded by $(d-k)^{\frac{d}{2}}$ through an application of Hadamard's determinant inequality.

Assuming for the moment that $[\mathbb{Z}^d:\Lambda] < \infty$, we deduce that $(n,0,\hdots,0) \in \Lambda$ for some natural number $n \leq (d-k)^{\frac{d}{2}}$. We see that $d-k$ must divide $n$ since $d-k$ divides the sum of the coordinates of every element of $\Lambda$. Hence we have $n = (d-k)l$ with $l \in \mathbb{N}$ bounded by $(d-k)^{\frac{d}{2}-1}$. The expression of $(n,0,\hdots,0)$ as a linear combination of the $A^i v$ is necessarily unique and the coefficients of the $A^i v$ in this linear combination can also be bounded in absolute value in terms of $k$ and $d$ only ($i=0,\hdots,d-1$). Translating all of this into terms of products of conjugates of $\pm \mathcal{H}^d$ yields that $(a\alpha_1^{d-k})^{l}$ can be written as a quotient of products of conjugates of $\pm \mathcal{H}^d$ for some natural number $l \leq (d-k)^{\frac{d}{2}-1}$, where the number of conjugates that appears is bounded in terms of $k$ and $d$ only as we wanted.

It remains to prove that $[\mathbb{Z}^d : \Lambda] < \infty$. Equivalently, we can show that the vector subspace $V$ of $\mathbb{C}^d$ generated by the $A^i v$ ($i=0,\hdots,d-1$) has dimension $d$. Over $\mathbb{C}$, the matrix $A$ is diagonalizable and we have $\mathbb{C}^d = \bigoplus_{i=0}^{d-1}{W_{\zeta^i}}$, where $\zeta$ is a primitive $d$-th root of unity and
\[ W_\lambda = \{w \in \mathbb{C}^d; Aw = \lambda w\} \quad (\lambda \in \mathbb{C}).\]

The vector subspace $V$ is $A$-invariant and so $V = \bigoplus_{i=0}^{d-1}{(V \cap W_{\zeta^i})}$. It cannot be contained in $W_1$ since $Av \neq v$ (here we use that $0 < k < d$), so there exists some $j \in \{1,\hdots,d-1\}$ with $V \cap W_{\zeta^j} \neq \{0\}$. As $\dim W_{\zeta^i} = 1$ for all $i$, it follows that $W_{\zeta^j} \subset V$. Since $V$ is defined over $\mathbb{Q}$, it follows by conjugating that $\bigoplus_{i=1}^{d-1}{W_{\zeta^i}} \subset V$. But $0 \neq \sum_{i=0}^{d-1}{A^i v} \in V \cap W_1$, so $W_1 \subset V$ as well. It follows that $V = \bigoplus_{i=0}^{d-1}{W_{\zeta^i}} = \mathbb{C}^d$.
\end{proof}

By adapting the proof of Theorem \ref{thm:heightmaintoo}(iii), we can now strengthen Lemma \ref{lem:usefullemma}.

\begin{lem}\label{lem:yetanotherusefullemma}
Let $d \in \mathbb{N}$, $\epsilon > 0$, and $k \in \{1,\hdots,d-1\}$. There exists a constant $C = C(k,d,\epsilon)$ such that
\begin{multline}
|\{\alpha \in A(k,d,\mathcal{H})\mbox{; the Galois group of the normal closure of $\mathbb{Q}(\alpha)$ acts} \nonumber\\
\mbox{$2$-transitively on the conjugates of $\alpha$}\}| \leq C\mathcal{H}^{\epsilon}\nonumber
\end{multline}
for all $\mathcal{H} \geq 1$.
\end{lem}

\begin{proof}
Let $\mathcal{H} \in [1,\infty)$ and $\alpha \in A(k,d,\mathcal{H})$ such that the Galois group of the normal closure of $\mathbb{Q}(\alpha)$ acts $2$-transitively on the conjugates of $\alpha$. Let $\alpha_1,\hdots,\alpha_d$ be the conjugates of $\alpha$. We want to mimick the proof of Theorem \ref{thm:heightmaintoo}(iii). We have $\mathcal{H}^d = \pm a\prod_{i \in I}{\alpha_i}$ for some $I \subset \{1,\hdots,d\}$ with $|I| = d-k$ and $a \in \mathbb{N}$ the leading coefficient of a minimal polynomial of $\alpha$ in $\mathbb{Z}[t]$. The Galois group of the normal closure of $\mathbb{Q}(\alpha)$ can be identified with a subgroup $G$ of the symmetric group $S_d$. To a (formal) product $\prod_{i=1}^{d}{\alpha_i^{e_i}}$ with $e_i \in \mathbb{Z}$ we again associate a vector $(e_1,\hdots,e_d) \in \mathbb{Z}^d$. The group $G$ then acts on $\mathbb{Z}^d$ by permuting the coordinates. We will denote the vector associated to $\prod_{i \in I}{\alpha_i}$ by $v$. As we have seen in the proof of Theorem \ref{thm:heightmaintoo}(iii), it suffices to show that the vector space $V$ generated over $\mathbb{Q}$ by the $gv$ for $g \in G$ must be $\mathbb{Q}^d$ in order to prove the lemma.

Certainly, this vector space is $G$-invariant. Since $G$ acts $2$-transitively, we know that there are only $4$ $G$-invariant vector subspaces of $\mathbb{Q}^d$, i.e. $\{0\}$, $\mathbb{Q}(1,1,1,\hdots,1)$,
\[\mathbb{Q}(1,-1,0,\hdots,0) \oplus \mathbb{Q}(0,1,-1,0,\hdots,0) \oplus \cdots \oplus \mathbb{Q}(0,\hdots,0,1,-1),\]
and $\mathbb{Q}^d$ (see \cite{MR0450380}, Exercise 2.6). We can immediately exclude the first two since neither of them contains the vector $v$. Furthermore, the vector $\sum_{g \in G}{gv}$ is non-zero and lies in $\mathbb{Q}(1,1,1,\hdots,1)$, so we can also exclude the third one. It follows that $V = \mathbb{Q}^d$ and we are done.
\end{proof}

The next theorem shows that the height function together with the degree is in some sense ``almost injective'' if the degree is at least $2$. 

\begin{thm}\label{thm:coarse}
Let $d \geq 2$. For every $\epsilon > 0$, there is $\mathcal{H}_0 = \mathcal{H}_0(d,\epsilon) \in \mathbb{R}$ such that
\[ \frac{|\{\alpha \in \mathbb{C};[\mathbb{Q}(\alpha):\mathbb{Q}] = d, H(\alpha) \leq \mathcal{H}\}|}{|\{H(\alpha);\alpha \in \mathbb{C}, [\mathbb{Q}(\alpha):\mathbb{Q}] = d, H(\alpha) \leq \mathcal{H}\}|} \leq \mathcal{H}^{\epsilon}\]
for all $\mathcal{H} \geq \mathcal{H}_0$.
\end{thm}

Theorem \ref{thm:coarse} is patently wrong for $d=1$, where the left-hand side of the inequality in the theorem grows linearly in $\mathcal{H}$.

\begin{proof}
First, we can replace the numerator in the inequality by the cardinality of the set
\begin{multline}
\{\alpha \in \mathbb{C};[\mathbb{Q}(\alpha):\mathbb{Q}] = d, H(\alpha) \leq \mathcal{H}, \mbox{ precisely one conjugate of }\alpha \nonumber\\
\mbox{ lies outside the open unit disk}\}.\nonumber
\end{multline}

Why? By Lemma \ref{lem:growthlemma}, the number of $\alpha$ of degree $d$ with precisely one conjugate outside the open unit disk and height at most $\mathcal{H}$ grows asymptotically like some constant positive multiple of $\mathcal{H}^{d(d+1)}$. Because of Lemma \ref{lem:growthlemma}, applied for all $k \in \{0,\hdots,d\}$ (or thanks to the main result of \cite{MR2487698}), demanding that $\alpha$ is in this set then changes the left-hand side of the inequality in the theorem by a factor bounded from below by some $c = c(d) > 0$ for $\mathcal{H}$ large enough in terms of $d$.

Let
\[B(d;\mathcal{H}) = \{H(\alpha);\alpha \in \mathbb{C}, [\mathbb{Q}(\alpha):\mathbb{Q}] = d, H(\alpha) \leq \mathcal{H}\},\]
then we can rewrite our new numerator as

\begin{multline}
\sum_{\tilde{\mathcal{H}} \in B(d;\mathcal{H})}|\{\alpha \in \mathbb{C};[\mathbb{Q}(\alpha):\mathbb{Q}] = d, H(\alpha) = \tilde{\mathcal{H}}, \mbox{ precisely one conjugate of }\alpha \nonumber\\
\mbox{ lies outside the open unit disk}\}|.\nonumber
\end{multline}

By Theorem \ref{thm:heightmaintoo}(i) each summand here is bounded by $C\mathcal{H}^{\frac{\epsilon}{2}}$ for some $C = C(d,\epsilon)$ and we are done.
\end{proof}

\section{The case $\gcd(k,d) = 1$}

In this section, we prove Theorem \ref{thm:unconditionalupperbound}, which will give a useful unconditional upper bound for $|A(k,d,\mathcal{H})|$. Theorem \ref{thm:unconditionalupperbound} also provides a further strengthening of Lemmas \ref{lem:usefullemma} and \ref{lem:yetanotherusefullemma}. We first prove an auxiliary lemma that will also be useful later.

\begin{lem}\label{lem:degovernormalclosure}
Let $k,d \in \mathbb{N}$ such that $0 < k < d$ and let $\mathcal{H} \in \bar{\mathbb{Q}} \cap [1,\infty)$. Let $K$ denote the normal closure of $\mathbb{Q}(\mathcal{H}^d)$ and suppose that $\alpha \in A(k,d,\mathcal{H})$. Then $[K(\alpha):K]$ divides $\gcd(k,d)$.

Furthermore, let $\alpha_1, \hdots, \alpha_d$ be the conjugates of $\alpha$, numbered so that $|\alpha_i| \geq 1$ if and only if $1 \leq i \leq d-k$. Then any element $\sigma \in \Gal(\bar{\mathbb{Q}}/\mathbb{Q})$ that fixes $\mathcal{H}^d$ also fixes the set $\{\alpha_1,\hdots,\alpha_{d-k}\}$.
\end{lem}

\begin{proof}
We first prove the second part of the lemma: If it were false, there would exist an element $\sigma \in \Gal(\bar{\mathbb{Q}}/\mathbb{Q})$ that fixes $\mathcal{H}^d$, but does not fix the set $\{\alpha_1,\hdots,\alpha_{d-k}\}$. But this immediately yields a contradiction since $\mathcal{H}^d = a\alpha_1\cdots\alpha_{d-k}$ for some non-zero integer $a$ and
\[ \left|\prod_{s \in S}{\alpha_s}\right| < |\alpha_1\cdots\alpha_{d-k}|\]
for every subset $S$ of $\{1,\hdots,d\}$ of cardinality $d-k$ that is not equal to $\{1,\hdots,d-k\}$. 

We deduce that the coefficients of the polynomial $\prod_{i=1}^{d-k}{(t-\alpha_i)}$ belong to $\mathbb{Q}(\mathcal{H}^d)$. In order to prove the first part of the lemma, we will make use of the following simple facts: If $K_2/K_1$ is a finite Galois extension of fields of characteristic $0$ within a fixed algebraic closure $\overline{K_1}$ and $\xi \in \overline{K_1}$, then $[K_2(\xi):K_2]$ divides $[K_1(\xi):K_1]$. Furthermore, if $\eta$ is a conjugate of $\xi$ over $K_1$, then $[K_2(\eta):K_2] = [K_2(\xi):K_2]$.

We deduce that $[K(\alpha):K] = [K(\alpha_i):K]$ divides $[\mathbb{Q}(\mathcal{H}^d,\alpha_i):\mathbb{Q}(\mathcal{H}^d)]$ ($i=1,\hdots,d-k$) and divides $[\mathbb{Q}(\alpha):\mathbb{Q}] = d$. But by the above, $d-k$ is the sum of some of the $[\mathbb{Q}(\mathcal{H}^d,\alpha_i):\mathbb{Q}(\mathcal{H}^d)]$, namely one for each irreducible factor of $\prod_{i=1}^{d-k}{(t-\alpha_i)}$ in $\mathbb{Q}(\mathcal{H}^d)[t]$. So $[K(\alpha):K]$ divides $d-k$ and $d$, hence divides $\gcd(k,d)$. This completes the proof of the lemma.
\end{proof}

We can now prove Theorem \ref{thm:unconditionalupperbound}.

\begin{thm}\label{thm:unconditionalupperbound}
Let $d \in \mathbb{N}$, $\epsilon > 0$, and $k \in \{1,\hdots,d-1\}$. There exists a constant $C$, depending only on $d$, $k$, and $\epsilon$, such that for all $\mathcal{H} \in \bar{\mathbb{Q}} \cap [1,\infty)$ the following hold:
\begin{enumerate}[label=(\roman*)]
\item Let $K$ denote the normal closure of $\mathbb{Q}(\mathcal{H}^d)$ and let $l \in \mathbb{N}$ divide $\gcd(k,d)$, then
\[ |\{\alpha \in A(k,d,\mathcal{H}); [K(\alpha):K] = l\}| \leq C\mathcal{H}^{d(l-1)+\epsilon},\]
\item $|A(k,d,\mathcal{H})| \leq C\mathcal{H}^{d(\gcd(k,d)-1)+\epsilon}$, and
\item $\begin{aligned}[t]
|\{\alpha \in A(k,d,\mathcal{H})\mbox{; the Galois group of the normal closure of $\mathbb{Q}(\alpha)$ acts} \\
\mbox{primitively on the set of conjugates of $\alpha$}\}| \leq C\mathcal{H}^{\epsilon}.
\end{aligned}$
\end{enumerate}
In particular, $a(k,d) = 0$ if $\gcd(k,d) = 1$.
\end{thm}

We will see later that the exponent in the bound for $|A(k,d,\mathcal{H})|$ is indeed sharp for every choice of $(k,d)$. Let us also note at this stage that one might hope a priori to prove that $a(k,d) = 0$ for all $d \in \mathbb{N}$ and $k \in \{1,\hdots,d-1\}$ with $\gcd(k,d) = 1$ by showing the following: For any transitive subgroup $G$ of the symmetric group $S_d$ and any vector $v \in \mathbb{Q}^d$ with exactly $k$ entries equal to $1$ and $d-k$ entries equal to $0$, the set $Gv$ generates $\mathbb{Q}^d$. Unfortunately, this statement is wrong. One can construct a counterexample with $G$ equal to the subgroup generated by the $d$-cycle $(12\cdots d)$ from any counterexample to the following statement: Any sum of $k$ distinct $d$-th roots of unity is non-zero. If we denote $e^{\frac{2\pi \sqrt{-1}}{n}}$ by $\zeta_n$ for $n \in \mathbb{N}$, then a construction by R\'{e}dei (see \cite{MR0103881}, Satz 9) yields counterexamples like
\[ 0 = (-1) + (-1)(-1) = \zeta_2 + \left(\sum_{i=1}^{2}{\zeta_3^i}\right)\left(\sum_{j=1}^{6}{\zeta_7^j}\right),\]
where the right-hand side is a sum of $13$ distinct $42$-nd roots of unity. If $G$ is a $2$-transitive subgroup of $S_d$, then it follows from the proof of Lemma \ref{lem:yetanotherusefullemma} that the statement is correct.

We see that Theorem \ref{thm:unconditionalupperbound}(i) yields an upper bound with exponent $\epsilon$ as soon as the normal closure of $\mathbb{Q}(\mathcal{H}^d)$ contains $\alpha$. If we restrict ourselves to $\alpha$ such that $\mathbb{Q}(\alpha)$ is Galois over $\mathbb{Q}$, we can for example obtain such a bound as soon as $[\mathbb{Q}(\mathcal{H}^d):\mathbb{Q}] = d$. In Theorem \ref{thm:shishi}, we will see another case where Theorem \ref{thm:unconditionalupperbound}(i) can be applied with $l = 1$.

\begin{proof}[Proof of Theorem \ref{thm:unconditionalupperbound}]
Let $\mathcal{H} \in \bar{\mathbb{Q}} \cap [1,\infty)$ and let $\alpha \in A(k,d,\mathcal{H})$. Let $K$ be the normal closure of $\mathbb{Q}(\mathcal{H}^d)$ and set $l = [K(\alpha):K]$. By Lemma \ref{lem:degovernormalclosure}, $l$ divides $\gcd(k,d)$. Thus, part (ii) from the theorem directly follows from part (i), after adjusting the constant $C$. Since $l \leq \gcd(k,d) < d$, we must have $l = 1$ if $\Gal(\bar{\mathbb{Q}}/\mathbb{Q})$ acts primitively on the set of conjugates of $\alpha$, so part (iii) also follows from part (i).

We now fix $l$ and prove part (i): Let $\alpha_1, \hdots, \alpha_d$ be the conjugates of $\alpha$, numbered so that $|\alpha_i| \geq 1$ if and only if $1 \leq i \leq d-k$. Let $a$ be the (non-zero) leading coefficient of a minimal polynomial of $\alpha$ in $\mathbb{Z}[t]$, chosen such that $H(\alpha)^d = a\alpha_1\cdots\alpha_{d-k}$. It follows from Lemma \ref{lem:degovernormalclosure} that
\[ N_{\mathbb{Q}(\mathcal{H}^d)/\mathbb{Q}}(\mathcal{H}^d) = a^{[\mathbb{Q}(\mathcal{H}^d):\mathbb{Q}]}\prod_{I \in \mathcal{I}}{\prod_{\beta \in I}{\beta}},\]
where $\mathcal{I}$ is the orbit of $\{\alpha_1,\hdots,\alpha_{d-k}\}$ under the Galois group of the normal closure of $\mathbb{Q}(\alpha)$ and the cardinality of $\mathcal{I}$ is $[\mathbb{Q}(\mathcal{H}^d):\mathbb{Q}]$. Since the Galois group acts transitively on $\{\alpha_1,\hdots,\alpha_d\}$, we have that

\begin{equation}\label{eq:normofhd}
N_{\mathbb{Q}(\mathcal{H}^d)/\mathbb{Q}}(\mathcal{H}^d) = a^{[\mathbb{Q}(\mathcal{H}^d):\mathbb{Q}]}(\alpha_1\cdots\alpha_d)^{\left(1-\frac{k}{d}\right)[\mathbb{Q}(\mathcal{H}^d):\mathbb{Q}]}.
\end{equation}

In particular, $d$ divides $(d-k)[\mathbb{Q}(\mathcal{H}^d):\mathbb{Q}]$. Since $k > 0$ and $a\alpha_1\cdots\alpha_d \in \mathbb{Z}$, we have that $a$ divides $N_{\mathbb{Q}(\mathcal{H}^d)/\mathbb{Q}}(\mathcal{H}^d)$ in $\mathbb{Z}$. So the number of possibilities for $a$ is bounded by $C_1\mathcal{H}^{\frac{\epsilon}{3}}$ for some constant $C_1$, depending only on $d$ and $\epsilon$. Hence we can assume that $a \in \mathbb{Z}\backslash\{0\}$ is fixed.

Let $I \subset \{\alpha_1,\hdots,\alpha_d\}$ be the subset of conjugates of $\alpha$ over $K$ (of cardinality $l$). For $j \in \{1,\hdots,l\}$, we set
\[ \gamma_j = a\sum_{J \subset I, |J| = j}{\prod_{\beta \in J}{\beta}}.\]
All the $\gamma_j$ lie in the fixed number field $K$ that is determined uniquely by $\mathcal{H}$ and $d$. We deduce from Lemma \ref{lem:yetmoreusefullemmata} that the $\gamma_j$ are algebraic integers ($j = 1,\hdots,l$).

The orbit of $I$ under $\Gal(\bar{\mathbb{Q}}/\mathbb{Q})$ consists of $\frac{d}{l}$ pairwise disjoint sets $I = I_1$, \dots, $I_{\frac{d}{l}}$. We calculate that $N_{K/\mathbb{Q}}(\gamma_j)$ is equal to
\[\left(a^{\frac{d}{l}}\prod_{s=1}^{\frac{d}{l}}{\sum_{J \subset I_s, |J| = j}{\prod_{\beta \in J}{\beta}}}\right)^{\frac{[K:\mathbb{Q}]l}{d}}.\]
By Lemma \ref{lem:yetmoreusefullemmata}, the number
\[N_j = a\prod_{s=1}^{\frac{d}{l}}{\sum_{J \subset I_s, |J| = j}{\prod_{\beta \in J}{\beta}}}\]
is a rational integer and together with $a$, it completely determines $N_{K/\mathbb{Q}}(\gamma_j)$.

If $j=l$, then $N_j = N_l$ divides $N_{\mathbb{Q}(\mathcal{H}^d)/\mathbb{Q}}(\mathcal{H}^d)$ by \eqref{eq:normofhd} since $k < d$. Therefore, $N_l$ is already determined up to $C_2\mathcal{H}^{\frac{\epsilon}{3}}$ possibilities, where $C_2$ depends only on $d$, $k$, and $\epsilon$. If $j \in \{1,\hdots,l-1\}$, then $N_j$ is at least bounded in absolute value by $C_3\mathcal{H}^d$, where $C_3$ depends only on $d$ and $k$.

The algebraic integers $\gamma_j$ ($j=1,\hdots,l$) lie in the given number field $K$ of degree at most $d!$ and their height is bounded by $C_4\mathcal{H}^d$, where $C_4$ depends only on $d$ and $k$. It therefore follows from Lemma \ref{lem:givennormboundedheight} that the number of possibilities for each of them, if their $K/\mathbb{Q}$-norm is fixed, is bounded by $C_5\mathcal{H}^{\frac{\epsilon}{3l}}$, where $C_5$ depends only on $d$, $k$, and $\epsilon$. Part (i) of the theorem now follows since $a\alpha^l + \sum_{j=1}^{l}{(-1)^j\gamma_j\alpha^{l-j}} = 0$ and so $\alpha$ is determined up to Galois conjugation by $l$, $a$, and the $\gamma_j$ ($j=1,\hdots,l$).
\end{proof}

\section{The case $(k,d) = (2,4)$}
One might be tempted to conjecture that $a(k,d) = 0$ for all $d \geq 2$ and $0 < k < d$, but this is not true. We begin our investigations by studying the simplest non-trivial case, namely $(k,d) = (2,4)$. In this case, there are three possibilities for $[\mathbb{Q}(\mathcal{H}^4):\mathbb{Q}]$, namely $2$, $4$, or $6$. In the last case, we can apply Lemma \ref{lem:usefullemma} to obtain that $|A(2,4,\mathcal{H})|$ grows more slowly than $\mathcal{H}^{\epsilon}$ for every $\epsilon > 0$. We now show in the next theorem that the same holds in the middle case, where $[\mathbb{Q}(\mathcal{H}^4):\mathbb{Q}] = 4$.

\begin{thm}\label{thm:shishi}
Let $\epsilon > 0$. There exists a constant $C = C(\epsilon)$ such that $|A(2,4,\mathcal{H})| \leq C\mathcal{H}^{\epsilon}$ for all $\mathcal{H} \geq 1$ with $[\mathbb{Q}(\mathcal{H}^4):\mathbb{Q}] = 4$.
\end{thm}

\begin{proof}
If $[\mathbb{Q}(\mathcal{H}^4):\mathbb{Q}] = 4$ and $\alpha \in A(2,4,\mathcal{H})$, then the normal closure of $\mathbb{Q}(\alpha)$ is either $\mathbb{Q}(\alpha)$ or a number field of degree $8$; otherwise, its Galois group would be naturally isomorphic to the symmetric or the alternating group on $4$ elements and we would get $[\mathbb{Q}(\mathcal{H}^4):\mathbb{Q}] = 6$ (recall that by Lemma \ref{lem:degovernormalclosure}, a Galois automorphism of $\bar{\mathbb{Q}}$ can only fix $\mathcal{H}^4$ if it fixes the set of conjugates of $\alpha$ that lie outside the open unit disk). We denote the normal closure of $\mathbb{Q}(\mathcal{H}^4)$ by $K$.

In the first case, i.e. if the normal closure of $\mathbb{Q}(\alpha)$ is $\mathbb{Q}(\alpha)$, $K$ coincides with the normal closure of $\mathbb{Q}(\alpha)$ as both are equal to $\mathbb{Q}(\alpha)$.

In the second case, i.e. if the normal closure of $\mathbb{Q}(\alpha)$ is a number field of degree $8$, the Galois group of the normal closure of $\mathbb{Q}(\alpha)$ is isomorphic to the dihedral group $D_4$ and $\mathbb{Q}(\mathcal{H}^4)$ is a quartic subfield of that normal closure. If $K$ is not equal to the normal closure of $\mathbb{Q}(\alpha)$, then the extension $\mathbb{Q}(\mathcal{H}^4)/\mathbb{Q}$ is Galois. Suppose now that the conjugates of $\alpha$ are the $\alpha_i$ ($i=1,\hdots,4$) and that the Galois group is generated by field automorphisms acting on the conjugates $\alpha_i$ by acting on their indices as the cycle $(1234)$ and the transposition $(13)$. Since $[\mathbb{Q}(\mathcal{H}^4):\mathbb{Q}] = 4$, we can assume after a cyclic renumbering that $\mathcal{H}^4 = \pm a \alpha_1 \alpha_2$, where $a \in \mathbb{N}$ is the leading coefficient of a minimal polynomial of $\alpha$ in $\mathbb{Z}[t]$. The only subfield of the normal closure of $\mathbb{Q}(\alpha)$ of degree $4$ that is Galois over $\mathbb{Q}$ corresponds under the Galois correspondence to the cyclic normal subgroup of $D_4$ generated by $(13)(24)$. But this element does not fix $\mathcal{H}^4$ since $|\alpha_1\alpha_2| \geq 1 > |\alpha_3\alpha_4|$. So $\mathbb{Q}(\mathcal{H}^4)$ cannot be Galois over $\mathbb{Q}$ and it follows also in this case that $K$ is equal to the normal closure of $\mathbb{Q}(\alpha)$.

The theorem now follows from Theorem \ref{thm:unconditionalupperbound}(i) with $l = 1$.
\end{proof}

In the case where $[\mathbb{Q}(\mathcal{H}^4):\mathbb{Q}] = 2$, it follows from Theorem \ref{thm:unconditionalupperbound} that we have $|A(2,4,\mathcal{H})| \leq C(\epsilon)\mathcal{H}^{4+\epsilon}$ for all such $\mathcal{H}$. However, the next theorem shows that one cannot always expect this growth and in fact one cannot obtain a uniform growth rate in $\mathcal{H}$ even after partitioning $A(2,4)$ into an arbitrary finite number of subsets. In Section \ref{sec:conditionallowerbound}, we will prove that $|A(2,4,\mathcal{H})| \geq C'\mathcal{H}^{4-\epsilon}$ if $[\mathbb{Q}(\mathcal{H}^4):\mathbb{Q}] = 2$ and there exists $\alpha \in A(2,4,\mathcal{H})$ satisfying a certain additional condition depending on $\mathbb{Q}(\mathcal{H}^4)$ and a parameter $\delta$ (cf. Theorem \ref{thm:conditionallowerboundsimplified}), but the constant $C'$ will also depend on $\mathbb{Q}(\mathcal{H}^4)$ and $\delta$.

\begin{thm}\label{thm:fail}
The limit $a(2,4)$ does not exist. For every $\kappa \in [0,4]$, there exists a sequence $(\mathcal{H}_n)_{n \in \mathbb{N}}$ in $B(2,4)$ such that $[\mathbb{Q}(\mathcal{H}_n^4):\mathbb{Q}] = 2$ for all $n \in \mathbb{N}$,
\[ \lim_{n \to \infty}{\mathcal{H}_n} = \infty,\]
and
\[ \lim_{n \to \infty}{\frac{\log |A(2,4,\mathcal{H}_n)|}{\log \mathcal{H}_n}} = \kappa.\]
\end{thm}

\begin{proof}
Let $\kappa \in [0,4]$. We fix $m \in \mathbb{N}$ prime with $m \not\equiv 1 \mod 4$ and denote its positive square root by $\sqrt{m}$. We define $\overline{u_1+u_2\sqrt{m}} = u_1-u_2\sqrt{m}$ ($u_1,u_2 \in \mathbb{Q}$).

If $\kappa < 4$, we apply a theorem of Chebyshev \cite{Chebyshev} (Bertrand's postulate) to find a prime number $b_2 \in \mathbb{N}$ such that 
\begin{equation}\label{eq:ineqb2}
m^{\frac{\kappa}{8-2\kappa}} \leq |b_2| \leq 2m^{\frac{\kappa}{8-2\kappa}}.
\end{equation}
We then set $\beta = b_1 + b_2\sqrt{m}$, where $b_1 \in \{[b_2\sqrt{m}],[b_2\sqrt{m}]+1\}$ is not divisible by $b_2$. After maybe replacing $\beta, b_1, b_2$ by $-\beta, -b_1, -b_2$, which preserves \eqref{eq:ineqb2}, we can assume that $0 < \bar{\beta} < 1$.

If $\kappa = 4$, we take $m = 2$ and $\beta = (3+2\sqrt{2})^r$ for some $r \in \mathbb{N}$. The integers $b_1, b_2$ are then defined by $\beta = b_1+b_2\sqrt{2}$. We automatically have that $0 < \bar{\beta} < 1$.

If $\kappa > 0$, we assume that
\begin{equation}\label{eq:betablocker}
|\beta| \geq 4\sqrt{m}+8
\end{equation}
by choosing $m$ or $r$ sufficiently large. If $\kappa = 0$, we assume that $m \geq 5$. We set $\mathcal{H} = |\beta|^{\frac{1}{4}}$.

We record that
\begin{equation}\label{eq:upperbound}
\mathcal{H}^4 = |\beta| \leq 3\sqrt{m}|b_2| \leq 6m^{\frac{1}{2}+\frac{\kappa}{8-2\kappa}} = 6m^{\frac{4}{2(4-\kappa)}} \quad (\kappa < 4)
\end{equation}
as well as
\[m^{\frac{4}{2(4-\kappa)}} = m^{\frac{\kappa}{8-2\kappa}+\frac{1}{2}} \leq |b_2|\sqrt{m} \leq |\beta| = \mathcal{H}^4  \quad (\kappa < 4)\]
because of \eqref{eq:ineqb2} and hence
\begin{equation}\label{eq:lowerbound}
m \leq \mathcal{H}^{2(4-\kappa)}  \quad (\kappa < 4).
\end{equation}

Let $\epsilon > 0$. We will prove that there exist positive constants $\mathcal{H}_0, c, \mathcal{H}_1, C$ such that the constants $\mathcal{H}_0$ and $c$ depend only on $\kappa$, the constants $\mathcal{H}_1$ and $C$ depend only on $\kappa$ and $\epsilon$,
\[ |A(2,4,\mathcal{H})| \leq C\mathcal{H}^{\kappa+\epsilon}\]
if $\mathcal{H} \geq \mathcal{H}_1$, and
\[ |A(2,4,\mathcal{H})| \geq c\mathcal{H}^{\kappa}\]
if $\mathcal{H} \geq \mathcal{H}_0$. The theorem then follows since $[\mathbb{Q}(\mathcal{H}^4):\mathbb{Q}] = 2$ and $\mathcal{H}$ tends to infinity as $m$ or $r$ respectively tend to infinity.

\subsubsection*{Upper bound}
We first prove the upper bound. Let $\alpha \in A(2,4,\mathcal{H})$. Let $\alpha_1, \hdots,\alpha_4$ be the conjugates of $\alpha$, ordered such that $|\alpha_1|,|\alpha_2| \geq 1$, and let $a > 0$ be the leading coefficient of a minimal polynomial of $\alpha$ in $\mathbb{Z}[t]$. It follows that $\beta = \pm \mathcal{H}^4 \in \{\pm a\alpha_1\alpha_2\}$.

Let $F$ be the fixed field of the stabilizer $H$ of $\{\alpha_1,\alpha_2\}$ in the Galois group $G$ of the normal closure of $\mathbb{Q}(\alpha)$. By Lemma \ref{lem:degovernormalclosure}, every $\sigma \in G$ which fixes $\beta$ must lie in $H$. Since the converse implication holds trivially, it follows that $F = \mathbb{Q}(\beta)$ and $\overline{\beta} \in \{\pm a\alpha_3\alpha_4\}$. We deduce from Lemma \ref{lem:yetmoreusefullemmata} that $a$ divides $a^2\prod_{j=1}^{4}{\alpha_j} = \beta\overline{\beta}$ in $\mathbb{Z}$. Since $|\overline{\beta}| < 1$, it follows from well-known bounds for the divisor function that the number of possibilities for $a$ is bounded by $C_1\mathcal{H}^{\frac{\epsilon}{4}}$ for a certain constant $C_1$ that depends only on $\epsilon$.

From now on, we assume that $a$ is fixed and count the number of possibilities for $\alpha$. We have $a\alpha_1^2-\gamma\alpha_1 \pm \beta = 0$, where $\gamma = a(\alpha_1+\alpha_2)$. For a given $\alpha_1$, there are exactly four possible $\alpha$. From now on, we assume that $\alpha = \alpha_1$. It then suffices to bound the number of possibilities for $\gamma$.

Now $\gamma$ lies in $F$, so $\gamma \in \mathbb{Q}(\beta)$. Furthermore, we have that $\gamma = a(\alpha_1+\alpha_2) \in \mathbb{Q}(\beta)$ is an algebraic integer by Lemma \ref{lem:yetmoreusefullemmata}. Since $m \not\equiv 1 \mod 4$, we have $\gamma \in \mathbb{Z}+\mathbb{Z}\sqrt{m}$, so $\gamma = c_1+c_2\sqrt{m}$ for some $c_1, c_2 \in \mathbb{Z}$. Let $\tilde{a} = \gcd(c_1,c_2,a)$, $\tilde{c}_1 = \tilde{a}^{-1}c_1$, and $\tilde{c}_2 = \tilde{a}^{-1}c_2$. By the usual bound for the divisor function, the number of possibilities for $\tilde{a}$ is bounded from above by $C_2a^{\frac{\epsilon}{16}} \leq C_2\mathcal{H}^{\frac{\epsilon}{4}}$ with a constant $C_2$ that depends only on $\epsilon$. In the following, we assume that $\tilde{a}$ is fixed.

As $\overline{\gamma} = a(\alpha_3+\alpha_4)$, the integer $c_1^2-mc_2^2 = \gamma\overline{\gamma}$ is divisible by $a$ thanks to Lemma \ref{lem:yetmoreusefullemmata}. It follows that $\tilde{a}^2\gcd(a,\tilde{a}^2)^{-1}(\tilde{c}_1^2 - m\tilde{c}_2^2)$ is divisible by $a' = a\gcd(a,\tilde{a}^2)^{-1}$. As $\tilde{a}^2\gcd(a,\tilde{a}^2)^{-1}$ and $a'$ are coprime, we deduce that $\tilde{c}_1^2 - m\tilde{c}_2^2$ is divisible by $a'$. By construction, we have that $\gcd(\tilde{c}_1,\tilde{c}_2,a') = 1$. It follows that $\tilde{c}_2 \neq 0$ unless $a' = 1$.

 Furthermore, we know that
\[ |\tilde{c}_2| = \tilde{a}^{-1}|c_2| \leq \frac{|\gamma|+|\overline{\gamma}|}{2\tilde{a}\sqrt{m}} \leq \frac{2|\beta|+2a}{2\tilde{a}\sqrt{m}} \leq \frac{2|\beta|}{\tilde{a}\sqrt{m}}\]
since $|\alpha_3|, |\alpha_4| < 1$, $|\alpha_1|, |\alpha_2| \geq 1$, and $a|\alpha_1\alpha_2| = \mathcal{H}^4 = |\beta|$. Thanks to \eqref{eq:upperbound} and \eqref{eq:lowerbound}, it follows that
\begin{equation}\label{eq:boundctildetwo}
|\tilde{c}_2| \leq 12\frac{m^{\frac{4}{2(4-\kappa)}}}{\tilde{a}\sqrt{m}} = 12\frac{m^{\frac{\kappa}{2(4-\kappa)}}}{\tilde{a}} \leq 12\frac{\mathcal{H}^\kappa}{\tilde{a}},
\end{equation}
at least if $\kappa < 4$. If $\kappa = 4$, the same follows from $|\beta| = \mathcal{H}^4$ and $\sqrt{2} \leq 12$.

For a given $\tilde{c}_2$, we have to bound the number of $\tilde{c}_1 \in \mathbb{Z}$ such that $|\tilde{c}_1-\tilde{c}_2\sqrt{m}| = \tilde{a}^{-1}|\bar{\gamma}| < 2a\tilde{a}^{-1}$ and $\tilde{c}_1^2 - m\tilde{c}_2^2 \equiv 0 \mod a'$. Set $\tilde{m} = \gcd(m,a')$, then $\tilde{m}$ is squarefree and must divide $\tilde{c}_1$. Furthermore, $\tilde{m}$ is uniquely determined by $m$, $a$, and $\tilde{a}$, so we can assume it fixed. We set $c_1' = \tilde{c}_1\tilde{m}^{-1}$, $m' = m\tilde{m}^{-1}$, and $a'' = a'\tilde{m}^{-1}$. It follows that $\tilde{m}c_1'^{2} \equiv m'\tilde{c}_2^2 \mod a''$. By construction, we have $\gcd(m',a'') = 1$. We also have $\gcd(\tilde{c}_2^2,a'') = 1$ since a common prime divisor of $a''$ and $\tilde{c}_2$ would have to divide $a'$ and therefore $\tilde{c}_1$, but $\gcd(\tilde{c}_1,\tilde{c}_2,a') = 1$. It follows that $\gcd(\tilde{m}c_1'^{2},a'') = 1$ as well.

The number of square roots modulo $a''$ of a number coprime to $a''$ is bounded by $2^{s+1}$, where $s$ is the number of distinct prime factors of $a''$. The number of $c_1'$ satisfying $|c_1'-\tilde{c}_2\sqrt{m}\tilde{m}^{-1}| = \tilde{m}^{-1}|\tilde{c}_1 - \tilde{c}_2\sqrt{m}| < 2a(\tilde{a}\tilde{m})^{-1}$ that lie in a given congruence class modulo $a''$ is at most $4\gcd(a,\tilde{a}^2)\tilde{a}^{-1}$ since $\gcd(a,\tilde{a}^2)\tilde{a}^{-1}$ is a natural number and $a''\gcd(a,\tilde{a}^2)\tilde{a}^{-1}=a(\tilde{a}\tilde{m})^{-1}$. It follows that the number of $c_1'$ for a given $\tilde{c}_2$ is at most $2^{s+3}\gcd(a,\tilde{a}^2)\tilde{a}^{-1}$. If $a'' \geq 3$, we have $s < \frac{7}{5}\frac{\log a''}{\log \log a''}$ by Th\'{e}or\`{e}me 11 in \cite{MR736719}. As the function $x \mapsto \frac{\log x}{\log \log x}$ is strictly monotonically increasing for $x \geq 16$, all natural numbers less than $16$ have at most $2$ distinct prime factors, and $a'' \leq a \leq \mathcal{H}^4$, we have
\[ s \leq \max\left\{2,\frac{\frac{28}{5}{ \log \mathcal{H}}}{\log\log\max\{3,\mathcal{H}\}}\right\}.\]

Recall that $\tilde{c}_2$ can only be $0$ if $a' = 1$, in which case $\gcd(a,\tilde{a}^2)\tilde{a}^{-1} = a\tilde{a}^{-1} \leq a^{\frac{1}{2}}$. Thanks to \eqref{eq:boundctildetwo} and the above, the number of possibilities for the pair $(\tilde{c}_1,\tilde{c}_2)$ is then bounded from above by
\begin{multline}
\left(24\frac{\mathcal{H}^\kappa}{\tilde{a}}\gcd(a,\tilde{a}^2)\tilde{a}^{-1}+a^{\frac{1}{2}}\right) \cdot 8 \cdot \max\left\{4,\mathcal{H}^{\frac{28}{5\log\log\max\{3,\mathcal{H}\}}}\right\} \nonumber\\
\leq (24\mathcal{H}^\kappa+\mathcal{H}^2) \cdot 8 \cdot \max\left\{4,\mathcal{H}^{\frac{28}{5\log\log\max\{3,\mathcal{H}\}}}\right\}.\nonumber
\end{multline}
If $\kappa \geq 2$, we can estimate $\mathcal{H}^2 \leq \mathcal{H}^{\kappa}$.

If $\kappa < 2$, we have to study more closely the case that $\tilde{c}_2 = 0$. We use that $a$ is the leading coefficient of a minimal polynomial of $\alpha$ in $\mathbb{Z}[t]$. If $\tilde{c}_2 = 0$ and $\gamma = c_1$, we can therefore conclude that $a$ divides all coefficients of the polynomial
\begin{multline}
(at^2-a(\alpha_1+\alpha_2)t+a\alpha_1\alpha_2)(at^2-a(\alpha_3+\alpha_4)t+a\alpha_3\alpha_4) = \nonumber\\
(at^2-c_1t\pm\beta)(at^2-c_1t\pm\bar{\beta}) \in \mathbb{Z}[t].\nonumber
\end{multline}
Here, the sign of $\beta$ is the same as that of $\overline{\beta}$. In particular, $a$ divides $c_1(\beta+\bar{\beta}) = 2b_1c_1$ as well as $\beta\bar{\beta} = b_1^2-mb_2^2$.

Set $\delta = \gcd(b_1,b_1^2-mb_2^2) = \gcd(b_1,mb_2^2)$. It follows from \eqref{eq:ineqb2} that $0 < |b_1| \leq |b_2|\sqrt{m}+1 \leq 2m^{\frac{4}{2(4-\kappa)}} + 1$. As $\kappa < 2$, this implies together with \eqref{eq:upperbound} that $0 < |b_1| < m$ for $\mathcal{H} \geq \mathcal{H}_1 = \mathcal{H}_1(\kappa)$. We assume from now on that $\mathcal{H} \geq \mathcal{H}_1$. Since $m$ is prime, it then follows that $\delta = \gcd(b_1,b_2^2)$. But $b_2$ is prime and does not divide $b_1$, so $\delta = 1$. Since any common divisor of $a$ and $b_1$ must also divide $\delta$, it follows that $\gcd(a,b_1) = 1$.

We deduce that $c_1$ must be divisible by $a\gcd(a,2)^{-1}$. Since $|c_1| = |\overline{\gamma}| < 2a$, there are at most $8$ possibilities for $c_1$.

Putting everything together, we obtain that the number of possibilities for $\alpha$ is bounded by
\[2 \cdot 4 \cdot C_1\mathcal{H}^{\frac{\epsilon}{4}} \cdot C_2 \mathcal{H}^{\frac{\epsilon}{4}} \cdot 25\mathcal{H}^\kappa \cdot 8 \cdot \max\{4,\mathcal{H}^{\frac{28}{5\log\log\max\{3,\mathcal{H}\}}}\} \leq C\mathcal{H}^{\kappa+\epsilon}\]
for $\mathcal{H} \geq \mathcal{H}_1$ with a constant $C$ that depends only on $\epsilon$.

\subsubsection*{Lower bound}
For the lower bound, we first treat the case $\kappa = 0$, so $\beta = \pm(\tilde{b}_1+2\sqrt{m})$ with $\tilde{b}_1 \in \{[2\sqrt{m}],[2\sqrt{m}]+1\}$ odd. Let $\sqrt{\beta}$ denote an arbitrary complex square root of $\beta$. The degree of $\sqrt{\beta}$ is $4$: Otherwise, $\sqrt{\beta}$ would have to be an element $a_1+a_2\sqrt{m}$ of $\mathbb{Z}[\sqrt{m}]$, which implies that $a_1^2+a_2^2m+2a_1a_2\sqrt{m} = \beta$, so $a_1, a_2 \in \{\pm 1\}$ and $m+1 \in \{\pm[2\sqrt{m}],\pm([2\sqrt{m}]+1)\}$. This yields a contradiction with $m \geq 5$. Therefore, we have $\sqrt{\beta} \in A(2,4)$. Since $H(\sqrt{\beta}) = \mathcal{H}$, the lower bound holds with $c = 1$.

We now assume that $\kappa > 0$. We choose $\gamma = c_1+c_2\sqrt{m}$ with
\[ c_2 \in \left\{1,\hdots,\left[\frac{|\beta|}{2\sqrt{m}}-\frac{2}{\sqrt{m}}\right]\right\}\]
and $c_1 = [c_2\sqrt{m}]+1$. It follows that 
\begin{equation}\label{eq:boundforgamma}
0 < \gamma \leq 2\sqrt{m}c_2+1 \leq |\beta| -3.
\end{equation}
We set $\alpha = \frac{\gamma}{2} + \sqrt{\frac{\gamma^2}{4}-\beta}$, so $\alpha^2-\gamma\alpha+\beta = 0$, where $\sqrt{\frac{\gamma^2}{4}-\beta}$ denotes the positive square root if $\frac{\gamma^2}{4}-\beta > 0$ and an arbitrary complex square root otherwise. It follows that $\alpha$ is an algebraic integer of degree dividing $4$. Note that $\alpha \neq 0$ and $\gamma = \frac{\beta+\alpha^2}{\alpha}$ is uniquely determined by $\alpha$.

We begin by controlling the cases where $[\mathbb{Q}(\alpha):\mathbb{Q}] < 4$.

 If $\alpha$ were a rational integer, then $\alpha$ would be a common divisor of $b_1$ and $b_2$. As $b_1$ and $b_2$ are coprime by construction, it would follow that $\alpha = \pm 1$ and therefore
 \[|\gamma| = \left|\frac{\beta+\alpha^2}{\alpha}\right| = |\beta+1| \geq |\beta|-1.\]
 This contradicts \eqref{eq:boundforgamma}. So $\alpha$ cannot be a rational integer.
 
 If $\alpha$ is quadratic, we have $\alpha \in \mathbb{Z}[\sqrt{m}]$ (if not, we could apply an automorphism of $\bar{\mathbb{Q}}$ that sends $\sqrt{m}$ to $-\sqrt{m}$, but leaves $\alpha$ unchanged to the defining equation of $\alpha$ and obtain a contradiction). Therefore, $\gamma^2-4\beta$ is a square in $\mathbb{Z}[\sqrt{m}]$. It follows that $(\gamma+\delta)(\gamma-\delta) = 4\beta$ for a certain $\delta \in \mathbb{Z}[\sqrt{m}]$. By using an elementary bound for the divisor function, we deduce from $|N_{\mathbb{Q}(\sqrt{m})/\mathbb{Q}}(4\beta)| \leq 16|\beta|$ that the norms of the ideals generated by $\gamma+\delta$ and $\gamma-\delta$ lie in a set of cardinality at most $C_3|\beta|^{\frac{\kappa}{32}}$ for some constant $C_3 = C_3(\kappa)$. Of course, these norms are also at most equal to $|N_{\mathbb{Q}(\sqrt{m})/\mathbb{Q}}(4\beta)| \leq 16|\beta|$. The number of ideals of norm $N$ in a quadratic number field is bounded by the number of natural numbers dividing $N$. It follows that the ideals themselves lie in a set of cardinality at most $C_4|\beta|^{\frac{\kappa}{16}}$ for a constant $C_4$ that depends only on $\kappa$, so we can assume them to be fixed.
 
 This determines $\gamma+\delta$ and $\gamma-\delta$ up to multiplication by a unit of $\mathbb{Z}[\sqrt{m}]$. This unit is of the form $\zeta u^{l}$, where $\zeta = \pm1$, $l \in \mathbb{Z}$, and $u$ is fixed (depending on $m$) and satisfies $H(u) > 1$, so $H(u) \geq h_{2} = \min\{ H(\xi); \xi \in \mathbb{C}, [\mathbb{Q}(\xi):\mathbb{Q}] \leq 2, H(\xi) > 1\} > 1$. Using the fact that $|\bar{\gamma}| = \left|[c_2\sqrt{m}]+1-c_2\sqrt{m}\right| < 1$ and $0 < \bar{\beta} < 1$ together with \eqref{eq:boundforgamma} and fundamental properties of the height, we can bound the height of $\gamma \pm \delta$ from above by
 \[ 2H(\gamma)H(\delta) = 2H(\gamma)H(\gamma^2-4\beta)^{\frac{1}{2}} \leq 4\sqrt{2}H(\gamma)^2H(\beta) = 4\sqrt{2}|\gamma||\beta|^{\frac{1}{2}} \leq 4\sqrt{2}|\beta|^\frac{3}{2}.\]
 
 If we write $\eta' = \eta \zeta u^l$, where $\eta$ and $\eta'$ are two possible values for $\gamma + \delta$, then it follows that $h_2^{|l|} \leq H(u)^{|l|} = H(u^l) \leq H(\eta)H(\eta') \leq 32|\beta|^3$ and so $|l|$ is bounded from above by $\frac{\log(32)+3\log |\beta|}{\log h_2}$. Hence there are at most $C_5\log|\beta|$ possibilities for the unit and hence for $\gamma + \delta$, where $C_5$ is an absolute constant. Now $\gamma + \delta$ determines $\gamma - \delta$ since $(\gamma+\delta)(\gamma-\delta) = 4\beta$ and $\beta$ is fixed. And $\gamma+\delta$ together with $\gamma-\delta$ determines $\gamma$, so there are at most $C_5\log|\beta|$ possibilities for $\gamma$ as well. It follows that $\alpha$ is quadratic for at most $C_6|\beta|^{\frac{\kappa}{8}} = C_6\mathcal{H}^{\frac{\kappa}{2}}$ choices of $\gamma$, where $C_6 = C_6(\kappa)$ depends only on $\kappa$.
 
Summarizing, we find that $\alpha$ has degree $< 4$ for at most $C_6\mathcal{H}^{\frac{\kappa}{2}}$ choices of $\gamma$.
 
If $\alpha$ has degree $4$ over $\mathbb{Q}$, which we from now on assume, its conjugates are $\frac{\gamma}{2}\pm\sqrt{\frac{\gamma^2}{4}-\beta}$ and $\frac{\bar{\gamma}}{2}\pm\sqrt{\frac{\bar{\gamma}^2}{4}-\bar{\beta}}$, where $\sqrt{\frac{\bar{\gamma}^2}{4}-\bar{\beta}}$ also denotes the positive square root if $\frac{\bar{\gamma}^2}{4}-\bar{\beta} > 0$ and an arbitrary complex square root otherwise.

If $\beta > 0$ and $\frac{\gamma^2}{4} < |\beta|$, then
\[ \left|\frac{\gamma}{2}\pm\sqrt{\frac{\gamma^2}{4}-\beta} \right| = |\beta|^{\frac{1}{2}} > 1.\]
If $\beta > 0$ and $\frac{\gamma^2}{4} \geq |\beta|$, we have
\[ \left|\frac{\gamma}{2}\pm\sqrt{\frac{\gamma^2}{4}-\beta} \right| \geq \frac{\gamma}{2}-\sqrt{\frac{\gamma^2}{4}-\beta} > 1,\]
since $\gamma < |\beta| +1$ and $\gamma \geq 2|\beta|^{\frac{1}{2}} \geq 2$.

If $\beta < 0$, we have
\[ \left|\frac{\gamma}{2}\pm\sqrt{\frac{\gamma^2}{4}-\beta} \right| \geq \sqrt{\frac{\gamma^2}{4}+|\beta|}-\frac{\gamma}{2} > 1,\]
since $\gamma < |\beta| -1$.

Recall that $\bar{\beta} > 0$. If $|\bar{\gamma}| < 2\bar{\beta}^{\frac{1}{2}}$, then $\sqrt{\frac{\bar{\gamma}^2}{4}-\bar{\beta}}$ is purely imaginary and
\[ \left|\frac{\bar{\gamma}}{2}\pm\sqrt{\frac{\bar{\gamma}^2}{4}-\bar{\beta}}\right| = \bar{\beta}^{\frac{1}{2}} < 1.\]
Otherwise, we have
\[ \left|\frac{\bar{\gamma}}{2}\pm\sqrt{\frac{\bar{\gamma}^2}{4}-\bar{\beta}}\right| \leq \left|\frac{\bar{\gamma}}{2}\right|+\left|\frac{\bar{\gamma}}{2}\right| =|\bar{\gamma}| = \left|[c_2\sqrt{m}]+1-c_2\sqrt{m}\right| < 1.\]

So in any case, $\alpha$ has two conjugates inside and two conjugates outside the open unit disk. Finally, we can compute that
\[ H(\alpha)^4 = \left|\left(\frac{\gamma}{2}+\sqrt{\frac{\gamma^2}{4}-\beta}\right)\left(\frac{\gamma}{2}-\sqrt{\frac{\gamma^2}{4}-\beta}\right)\right| = |\beta|,\]
so $H(\alpha) = |\beta|^{\frac{1}{4}} = \mathcal{H}$.

Thanks to \eqref{eq:betablocker}, the number of choices for $\gamma$ can be estimated as
\[\left[\frac{|\beta|}{2\sqrt{m}}-\frac{2}{\sqrt{m}}\right] \geq \frac{|\beta|-2\sqrt{m}-4}{2\sqrt{m}} \geq \frac{|\beta|}{4\sqrt{m}} = \frac{\mathcal{H}^4}{4\sqrt{m}}.\]
We can then use \eqref{eq:lowerbound} to deduce that the number of choices for $\gamma$ is equal to at least $\frac{\mathcal{H}^{\kappa}}{4}$ if $\kappa < 4$. If $\kappa = 4$, we get that the number of choices for $\gamma$ is equal to at least $\frac{\mathcal{H}^\kappa}{4\sqrt{2}}$. Since $\gamma$ is uniquely determined by $\alpha$, the lower bound is proven with $c = \frac{1}{8\sqrt{2}}$ and $\mathcal{H}_0 = (8\sqrt{2}C_6)^{\frac{2}{\kappa}}$.
\end{proof}

\section{The case $\gcd(k,d) > 1$}\label{sec:conditionallowerbound}

In fact, the situation is even worse than Theorem \ref{thm:fail} suggests: The limit $a(k,d)$ never exists if $0 < k < d$ and $\gcd(k,d) > 1$. In this section, we consider the general case where $\gcd(k,d) > 1$ and first prove the following more refined result about the frequency of the corresponding height values. It is valid for all $k \in \{1,\hdots,d-1\}$, but of interest mostly in the case where $\gcd(k,d) > 1$.

\begin{thm}\label{thm:conditionallowerbound}
Let $k, d \in \mathbb{N}$ such that $0 < k < d$. Let $\delta \in (0,1)$ and $\epsilon > 0$ and let $K \subset \bar{\mathbb{Q}}$ be a fixed Galois extension of $\mathbb{Q}$. Let $\mathcal{H} \in \bar{\mathbb{Q}} \cap [1,\infty)$ such that the normal closure of $\mathbb{Q}(\mathcal{H}^d)$ is equal to $K$.

For a subfield $L \subset K$ such that $[L:\mathbb{Q}] \mid d \mid [L:\mathbb{Q}]\gcd(k,d)$ and $\beta \in L$, set
\[ A_{L,\beta}(k,d,\mathcal{H}) = \{ \alpha \in A(k,d,\mathcal{H}); \mathbb{Q}(\alpha) \cap K = L, N_{\mathbb{Q}(\alpha)/L}(\alpha) = \beta\}. \]

Then we have that
\begin{equation}\label{eq:bigunion}
A(k,d,\mathcal{H}) = \bigcup_{\stackrel{L \subset K}{[L:\mathbb{Q}] | d | [L:\mathbb{Q}]\gcd(k,d)}}\bigcup_{\beta \in L}{A_{L,\beta}(k,d,\mathcal{H})}.
\end{equation}

There exists a constant $C = C(k,d,K,\delta,\epsilon) > 0$ such that if $A_{L,\beta}(k,d,\mathcal{H}) \neq \emptyset$ for a subfield $L \subset K$ as above and $\beta \in L$ and if furthermore for every field embedding $\sigma: \mathbb{Q}(\beta) \hookrightarrow \mathbb{C}$, we have either $|\sigma(\beta)| \geq (1-\delta)^{-1}$ or $|\sigma(\beta)| \leq 1-\delta$, then
\begin{equation}\label{eq:awesomelowerbound}
|A_{L,\beta}(k,d,\mathcal{H})| \geq C\mathcal{H}^{d(l-1)-\epsilon},
\end{equation}
where $l = d[L:\mathbb{Q}]^{-1}$.
\end{thm}

Before the proof, we make some remarks on this theorem: The number of possibilities for $L$ given $K$ is bounded in terms of $d$ and $k$. As $\beta$ is a product of $l$ conjugates of $\alpha$, we can bound its height by $\mathcal{H}^l$. Since $K/\mathbb{Q}$ is a Galois extension, we have that $[K(\alpha):K] = [\mathbb{Q}(\alpha):L] = l$ for any $\alpha \in A_{L,\beta}(k,d,\mathcal{H})$. An upper bound for $|A_{L,\beta}(k,d,\mathcal{H})|$ of the same growth order as \eqref{eq:awesomelowerbound} (up to $\mathcal{H}^{2\epsilon}$) is therefore provided by Theorem \ref{thm:unconditionalupperbound}(i). However, the following examples show that it is not possible in general to prove the lower bound \eqref{eq:awesomelowerbound} with $C$ depending on $k$, $d$, $\delta$, and $\epsilon$, but not on $K$, or with $C$ depending on $k$, $d$, $K$, and $\epsilon$, but not on $\delta$. For reasons of space, we grudgingly leave it to the reader to work out the details in the examples.

The necessity of the dependence on $K$ is shown by the following example:
\begin{ex}
Let $m \in \mathbb{N}$ be even such that $m-1$ and $m+1$ are both squarefree and $m > 2$. The asymptotic count of squarefree integers shows that there exist arbitrarily large such $m$. Set $\alpha = \sqrt{m+\sqrt{m^2-1}}$, where $\sqrt{\cdot}$ denotes the positive square root. One can show that $\alpha^2$ is not a square in $\mathbb{Q}(\sqrt{m^2-1})$ and so $[\mathbb{Q}(\alpha):\mathbb{Q}] = 4$. We find that $\alpha \in A(2,4,\mathcal{H})$ with $\mathcal{H}^4 = m+\sqrt{m^2-1}$. We have $K = L = \mathbb{Q}(\sqrt{m^2-1})$, $l = 2$, and $\beta = N_{\mathbb{Q}(\alpha)/L}(\alpha) = -(m+\sqrt{m^2-1})$ in Theorem \ref{thm:conditionallowerbound}. We can take $\delta = \frac{1}{2}$ for $m$ large enough. One can show that any $\alpha' \in A(2,4,\mathcal{H})$ with $|\alpha'| \geq 1$ is an algebraic integer and satisfies an equation $\alpha'^2+\gamma\alpha'\pm\beta= 0$ with $\gamma \in \mathcal{O}_K$. Let $\bar{\gamma}$ denote the image of $\gamma$ under the non-trivial field automorphism of $K$. Then one can show that $|\gamma| \leq 2|\beta|$ while $|\bar{\gamma}| \leq 2$. This implies that the number of such $\gamma$ is bounded independently of $m$, but $\mathcal{H} \to \infty$ as $m \to \infty$.
\end{ex}

The necessity of the dependence on $\delta$ is shown by the following example:
\begin{ex}
Let $(a,b) \in \mathbb{N}^2$ be a solution to $a^2 - 2b^2 = -1$ and set $\alpha = \sqrt{\frac{1+b\sqrt{2}}{a}}$, where $\sqrt{\cdot}$ again denotes the positive square root. The $\mathbb{Q}(\sqrt{2})/\mathbb{Q}$-norm of $\alpha^2$ is $-1$, which implies that $\alpha \not\in \mathbb{Q}(\sqrt{2})$ and so $[\mathbb{Q}(\alpha):\mathbb{Q}] = 4$. We find that a minimal polynomial of $\alpha$ in $\mathbb{Z}[t]$ is $at^4-2t^2-a$ and $\alpha \in A(2,4,\mathcal{H})$ with $\mathcal{H}^4 = 1+b\sqrt{2}$. We have $K = L = \mathbb{Q}(\sqrt{2})$, $l = 2$, and $\beta = N_{\mathbb{Q}(\alpha)/L}(\alpha) = -\frac{1+b\sqrt{2}}{a}$ in Theorem \ref{thm:conditionallowerbound}. Let $\alpha' \in A(2,4,\mathcal{H})$ such that $|\alpha'| \geq 1$ and let $a' > 0$ be the leading coefficient of a minimal polynomial of $\alpha'$ in $\mathbb{Z}[t]$. One can show that $a'$ divides $N_{K/\mathbb{Q}}(\mathcal{H}^4) = - a^2$ and that $b\sqrt{2}-1 < a' \leq b\sqrt{2}+1$. We deduce that $a' = a$. This implies that $\alpha'$ satisfies an equation $a\alpha'^2+\gamma\alpha'\pm(1+b\sqrt{2}) = 0$ with $\gamma \in \mathcal{O}_K$. Since $N(\mathcal{I}) = a$ for $\mathcal{I} = a\mathcal{O}_K + (1+b\sqrt{2})\mathcal{O}_K$, we must have $\gamma \in \mathcal{I}$. Let $\bar{\gamma}$ denote the image of $\gamma$ under the non-trivial field automorphism of $K$. Then one can show that $\max\{|\gamma|,|\bar{\gamma}|\} \leq 2(1+b\sqrt{2})$ while
\[ \min\{|\gamma|,|\bar{\gamma}|\} \leq \max\{1+b\sqrt{2}-a,a-b\sqrt{2}+1\} \leq 2.\]
Applying Theorem 2.1 in \cite{Widmer} with $\mathcal{S} = ((1,1),(1,1))$ and $C = \{(0,0)\}$ to the image of $\mathcal{I}$ under a Minkowski embedding (cf. the proof of Lemma \ref{lem:hilfssatzhilfe} below and note that the $K/\mathbb{Q}$-norm of every element of $\mathcal{I}$ is divisible by $a$) shows that the number of such $\gamma$ is bounded independently of $(a,b)$, but $\mathcal{H} \to \infty$ as $b \to \infty$.
\end{ex}

We now prove Theorem \ref{thm:conditionallowerbound}.

\begin{proof}[Proof of Theorem \ref{thm:conditionallowerbound}]\let\qed\relax
We first prove \eqref{eq:bigunion}: Let $\alpha \in A(k,d,\mathcal{H})$. The normal closure of $\mathbb{Q}(\mathcal{H}^d)$ is equal to $K$. Set $l = [K(\alpha):K]$. By Lemma \ref{lem:degovernormalclosure}, $l$ divides $\gcd(k,d)$. Set $L = K \cap \mathbb{Q}(\alpha)$. Since $K/\mathbb{Q}$ is Galois, we have that $[\mathbb{Q}(\alpha):L] = [K(\alpha):K] = l$. So $[L:\mathbb{Q}] = \frac{d}{l}$ divides $d$ and is divisible by $\frac{d}{\gcd(k,d)}$. We set $\beta = N_{\mathbb{Q}(\alpha)/L}(\alpha) \in L$ and it follows that $\alpha \in A_{L,\beta}(k,d,\mathcal{H})$. This proves \eqref{eq:bigunion}.

Next, we prove \eqref{eq:awesomelowerbound}: We fix $L$ and $\beta$ and suppose that $A_{L,\beta}(k,d,\mathcal{H}) \neq \emptyset$ and that for every field embedding $\sigma: \mathbb{Q}(\beta) \hookrightarrow \mathbb{C}$, we have either $|\sigma(\beta)| \geq (1-\delta)^{-1}$ or $|\sigma(\beta)| \leq 1-\delta$. It follows that there exists some $\alpha \in A_{L,\beta}(k,d,\mathcal{H})$.

Let $a$ denote the leading coefficient of a minimal polynomial of $\alpha$ in $\mathbb{Z}[t]$, chosen such that $a > 0$. If $P$ denotes the (monic) minimal polynomial of $\alpha$ in $L[t]$, then Lemma \ref{lem:yetmoreusefullemmata} shows that $aP \in \mathcal{O}_L[t]$. Let $\mathcal{I}$ denote the ideal of $\mathcal{O}_L$ generated by the coefficients of $aP$.

Since $K/\mathbb{Q}$ is Galois and $L \subset K$, every field embedding $\sigma: L \hookrightarrow \mathbb{C}$ factors through $K$. Thus, we can set
\[ \mathcal{J} = \prod_{\sigma: L \hookrightarrow \mathbb{C}}{\sigma(\mathcal{I})\mathcal{O}_K},\]
it is an ideal of $\mathcal{O}_K$. Since $\mathbb{Q}(\alpha) \supset L$, we have that $\prod_{\sigma: L \hookrightarrow \mathbb{C}}{\sigma(P)}$ is the minimal polynomial of $\alpha$ in $\mathbb{Q}[t]$ and $a\prod_{\sigma: L \hookrightarrow \mathbb{C}}{\sigma(P)}$ is a minimal polynomial of $\alpha$ in $\mathbb{Z}[t]$. In particular, the sets of complex zeroes of the $\sigma(P)$ form a partition of the conjugates of $\alpha$. This implies together with Lemma \ref{lem:yetmoreusefullemmata} that the ideal $\mathcal{J}$ is divisible by $a^{[L:\mathbb{Q}]-1}\mathcal{O}_K$. At the same time, $\mathcal{J}$ contains $a^{[L:\mathbb{Q}]-1}q$ for every coefficient $q$ of a minimal polynomial of $\alpha$ in $\mathbb{Z}[t]$. It follows that $\mathcal{J} = a^{[L:\mathbb{Q}]-1}\mathcal{O}_K$, which implies that
\begin{equation}\label{eq:normofmathcali}
N(\mathcal{I}) = a^{[L:\mathbb{Q}]-1}.
\end{equation}

Let $\alpha_1, \hdots, \alpha_d$ be the conjugates of $\alpha$, numbered so that $|\alpha_i| \geq 1$ if and only if $1 \leq i \leq d-k$. We deduce that $H(\alpha)^d = \pm a\alpha_1\cdots\alpha_{d-k}$. It follows from Lemma \ref{lem:degovernormalclosure} that the coefficients of the polynomial $\prod_{i=1}^{d-k}{(t-\alpha_i)}$ belong to $\mathbb{Q}(\mathcal{H}^d)$. By the proof of \eqref{eq:bigunion}, we have $l = [\mathbb{Q}(\alpha):L] = [K(\alpha):K]$. In particular, $P$ is also the minimal polynomial of $\alpha$ in $K[t]$. Let $I \subset \{\alpha_1,\hdots,\alpha_d\}$ be the subset of conjugates of $\alpha$ over $K$ (of cardinality $l$). The orbit of $I$ under $\Gal(\bar{\mathbb{Q}}/\mathbb{Q})$ consists of $\frac{d}{l}$ pairwise disjoint sets $I = I_1$, \dots, $I_{\frac{d}{l}}$. As the coefficients of the polynomial $\prod_{i=1}^{d-k}{(t-\alpha_i)}$ belong to $K$ and $K/\mathbb{Q}$ is Galois, we must have
\[ \{\alpha_1,\hdots,\alpha_{d-k}\} = \bigcup_{j \in S} I_j \]
for some $S \subset \left\{1,\hdots,\frac{d}{l}\right\}$.

This implies that for every $\sigma: L \hookrightarrow \mathbb{C}$, $\sigma(P)$ has either all complex zeroes inside or all complex zeroes outside the open unit disk. As $\beta$ is the constant coefficient of $P$ up to sign and $\prod_{\sigma: L \hookrightarrow \mathbb{C}}{\sigma(P)}$ is the minimal polynomial of $\alpha$ in $\mathbb{Q}[t]$, we have that
\begin{equation}\label{eq:heightintermsofbeta}
\mathcal{H}^d = a\prod_{\stackrel{\sigma: L \hookrightarrow \mathbb{C}}{|\sigma(\beta)| \geq 1}}{|\sigma(\beta)|}.
\end{equation}
We can also deduce that $|\sigma(\beta)| < 1$ for precisely $\frac{k}{l}$ embeddings $\sigma: L \hookrightarrow \mathbb{C}$.

Recall that $l = d[L:\mathbb{Q}]^{-1}$. We can and will assume without loss of generality that $l \geq 2$.

For $\underline{\gamma} = (\gamma_1,\hdots,\gamma_{l-1}) \in \mathcal{I}^{l-1}$, define the polynomials
\[ P_{\underline{\gamma}}(t) = at^l + \gamma_{l-1}t^{l-1} + \cdots + \gamma_1t + (-1)^la\beta \in \mathcal{O}_L[t]\]
and $Q_{\underline{\gamma}} = \prod_{\sigma: L \hookrightarrow \mathbb{C}}{\sigma(P_{\underline{\gamma}})} \in \mathbb{Z}[t]$.

Suppose that $\underline{\gamma}$ satisfies the following:
\begin{enumerate}
\item $\mathcal{I}$ is generated by $a$, $\gamma_1$, and $a\beta$,
\item $L = \mathbb{Q}(\gamma_1)$,
\item $|\sigma(\gamma_i)| \leq \frac{a\delta}{l}$ for all $\sigma: L \hookrightarrow \mathbb{C}$ such that $|\sigma(\beta)| < 1$ ($i=1,\hdots,l-1$),
\item $|\sigma(\gamma_i)| \leq \frac{a\delta|\sigma(\beta)|}{l}$ for all $\sigma: L \hookrightarrow \mathbb{C}$ such that $|\sigma(\beta)| \geq 1$ ($i=1,\hdots,l-1$), and
\item $P_{\underline{\gamma}}$ is irreducible in $K[t]$.
\end{enumerate}

It follows from (1), \eqref{eq:normofmathcali}, and the Gauss lemma (Lemma 1.6.3 in \cite{MR2216774}) that $Q_{\underline{\gamma}} = a^{[L:\mathbb{Q}]-1}Q'_{\underline{\gamma}}$, where $Q'_{\underline{\gamma}} \in \mathbb{Z}[t]$ is primitive with leading coefficient $a$.

If $|\sigma(\beta)| < 1$ for $\sigma: L \hookrightarrow \mathbb{C}$ and $\alpha_{\underline{\gamma},\sigma}$ is some complex zero of $\sigma(P_{\underline{\gamma}})$, then it follows from (3) that
\begin{equation}\label{eq:someinside}
a|\alpha_{\underline{\gamma},\sigma}|^l \leq \left(a|\sigma(\beta)|+\sum_{i=1}^{l-1}{|\sigma(\gamma_i)|}\right)\max\{1,|\alpha_{\underline{\gamma},\sigma}|\}^{l-1} < a(|\sigma(\beta)|+\delta)\max\{1,|\alpha_{\underline{\gamma},\sigma}|\}^{l-1}.
\end{equation}
As $|\sigma(\beta)| \leq 1-\delta$ by our hypothesis, this implies that $|\alpha_{\underline{\gamma},\sigma}| < 1$.

If $|\sigma(\beta)| \geq 1$ for $\sigma: L \hookrightarrow \mathbb{C}$ and $\alpha_{\underline{\gamma},\sigma}$ is some complex zero of $\sigma(P_{\underline{\gamma}})$, then it follows from (4) that
\begin{equation}\label{eq:othersoutside}
a|\sigma(\beta)| \leq \left(a+\sum_{i=1}^{l-1}{|\sigma(\gamma_i)|}\right)\max\{1,|\alpha_{\underline{\gamma},\sigma}|\}^{l} < a(1+\delta|\sigma(\beta)|)\max\{1,|\alpha_{\underline{\gamma},\sigma}|\}^{l}.
\end{equation}
As $|\sigma(\beta)| \geq (1-\delta)^{-1}$ by our hypothesis, this implies that $|\alpha_{\underline{\gamma},\sigma}| \geq 1$.

It follows from (2) that the $\sigma(P_{\underline{\gamma}})$ for $\sigma: L \hookrightarrow \mathbb{C}$ are pairwise distinct. Together with (5) and the fact that $K/\mathbb{Q}$ is Galois, this implies that $Q'_{\underline{\gamma}}$ is irreducible in $\mathbb{Q}[t]$ and therefore in $\mathbb{Z}[t]$.

Let $\alpha_{\underline{\gamma}}$ be a complex zero of $P_{\underline{\gamma}}$. It follows that $[\mathbb{Q}(\alpha_{\underline{\gamma}}):\mathbb{Q}] = [L:\mathbb{Q}]l = d$ and $H(\alpha_{\underline{\gamma}})^d = a\prod_{\stackrel{\sigma: L \hookrightarrow \mathbb{C}}{|\sigma(\beta)| \geq 1}}{|\sigma(\beta)|}$, which equals $\mathcal{H}^d$ by \eqref{eq:heightintermsofbeta}. Since $|\sigma(\beta)| < 1$ for precisely $\frac{k}{l}$ embeddings $\sigma: L \hookrightarrow \mathbb{C}$, we have that precisely $k$ conjugates of $\alpha_{\underline{\gamma}}$ lie inside the open unit disk.

Furthermore, we deduce from (5) that
\[ [L(\alpha_{\underline{\gamma}}):\mathbb{Q}] = [L(\alpha_{\underline{\gamma}}):L][L:\mathbb{Q}] = l[L:\mathbb{Q}] = [\mathbb{Q}(\alpha_{\underline{\gamma}}):\mathbb{Q}],\]
which implies that $\mathbb{Q}(\alpha_{\underline{\gamma}}) \supset L$. It follows that $L \subset \mathbb{Q}(\alpha_{\underline{\gamma}}) \cap K$. Since $[\mathbb{Q}(\alpha_{\underline{\gamma}}) : \mathbb{Q}(\alpha_{\underline{\gamma}}) \cap K] = [K(\alpha_{\underline{\gamma}}):K]$ and $[K(\alpha_{\underline{\gamma}}):K] = l = [\mathbb{Q}(\alpha_{\underline{\gamma}}):L]$ by (5), we must have $L = \mathbb{Q}(\alpha_{\underline{\gamma}}) \cap K$. We also have that $N_{\mathbb{Q}(\alpha_{\underline{\gamma}})/L}(\alpha_{\underline{\gamma}}) = \beta$ so that $\alpha_{\underline{\gamma}} \in A_{L,\beta}(k,d,\mathcal{H})$. Since $\underline{\gamma}$ is uniquely determined by $\alpha_{\underline{\gamma}}$, $a$, and $L$, we have reduced the proof of \eqref{eq:awesomelowerbound} to proving the following Lemma \ref{lem:hilfssatzhilfe}:
\end{proof}

\begin{lem}\label{lem:hilfssatzhilfe}
In the above setting, there exists a constant $C = C(k,d,K,\delta,\epsilon) > 0$ such that the number of $\underline{\gamma} \in \mathcal{I}^{l-1}$ satisfying (1) to (5) is greater than or equal to $C\mathcal{H}^{d(l-1)-\epsilon}$.
\end{lem}

\begin{proof}
We will use $c_1, c_2, \hdots$ for positive constants that depend only on $k$, $d$, $K$, $\delta$, and $\epsilon$. Recall that we have assumed that $l \geq 2$. We can assume without loss of generality that $\epsilon < \frac{1}{2}$.

We want to use Theorem 2.1 in \cite{Widmer}. Let $r$ and $s$ denote the number of real embeddings and pairs of complex conjugate embeddings of $L$ respectively. For each pair of complex conjugate embeddings of $L$, we choose one element of the pair. Furthermore, we order both the real embeddings of $L$ and the pairs of complex conjugate embeddings of $L$ in fixed ways each. Let $\Lambda$ denote the image of $\mathcal{I}$ inside $\mathbb{R}^r \times \mathbb{C}^s$ under the thus obtained Minkowski embedding. We identify $\mathbb{C}$ with $\mathbb{R}^2$ by identifying $(v,w) \in \mathbb{R}^2$ with $v+w\sqrt{-1} \in \mathbb{C}$ and we identify each $\gamma_i$ with its image in $\Lambda$ ($i=1,\hdots,l-1$). In the following, we use the notation of \cite{Widmer}: We set $n = r+s$, $N = [L:\mathbb{Q}]$, $C = \{\underline{0}\} \subset \mathbb{R}^N$, $m_j = \beta_j = 1$ for $1 \leq j \leq r$, and $m_j = \beta_j = 2$ for $r+1 \leq j \leq r+s$. For each $j \in \{1,\hdots,n\}$, let $\sigma_j: L \hookrightarrow \mathbb{C}$ denote the associated embedding used to define the Minkowski embedding.

We have $\Nm_{\mathbf{\beta}}(\Lambda) \geq a^{1-\frac{1}{[L:\mathbb{Q}]}} = a^{1-\frac{l}{d}}$ since the $L/\mathbb{Q}$-norm of any non-zero element of $\mathcal{I}$ is non-zero and divisible by $N(\mathcal{I}) = a^{[L:\mathbb{Q}]-1}$. The same lower bound holds for $\mu(\Lambda,B)$, where $B > 0$ is arbitrary. We set $Q_j = \frac{a\delta}{l}$ if $|\sigma_j(\beta)| < 1$ and $Q_j = \frac{a\delta|\sigma_j(\beta)|}{l}$ otherwise ($j=1,\hdots,n$). Conditions (3) and (4) for a fixed $i$ define a product $Z_{\mathbf{Q}}$ of intervals and disks satisfying conditions (1) and (2) on p. 480 of \cite{Widmer} for our choice of $Q_j$, $\mathbf{y}_j = \underline{0} \in \mathbb{R}^{m_j}$, and $(\kappa,M) = (8N^{5/2},1)$ (cf. \cite{Widmer}, p. 479).

Thanks to \eqref{eq:heightintermsofbeta}, we have $\overline{Q} = a^{1-\frac{l}{d}}\delta\mathcal{H}^{l}l^{-1}$. It follows from \eqref{eq:heightintermsofbeta} that the volume of $Z_{\mathbf{Q}}$ is greater than or equal to $c_1a^{[L:\mathbb{Q}]-1}\mathcal{H}^d$. Let $\Delta_L$ denote the discriminant of $L$, then the determinant of $\Lambda$ is equal to $2^{-s}|\Delta_L|^{\frac{1}{2}}N(\mathcal{I}) = 2^{-s}|\Delta_L|^{\frac{1}{2}}a^{[L:\mathbb{Q}]-1}$.

Theorem 2.1 in \cite{Widmer} then yields a main term which is greater than or equal to $c_2\mathcal{H}^d$ for the number of $\gamma_i$ satisfying (3) and (4), for fixed $i$. Choosing $B = Q_{\mathrm{max}}$ and using our lower bound for $\mu(\Lambda,B)$, we find that the corresponding error term is bounded from above by $c_3\mathcal{H}^{d-l}$.

We turn to (1). The ideal $\mathcal{I}' = a\mathcal{O}_L + a\beta\mathcal{O}_L$ is contained in $\mathcal{I}$ and $[\mathcal{I}:\mathcal{I}']$ divides $[\mathcal{I}:a\mathcal{O}_L] = a^{[L:\mathbb{Q}]}N(\mathcal{I})^{-1} = a$. If (1) is not satisfied, then $\gamma_1$ is contained in $\mathcal{I}\mathcal{P}$ for some prime ideal $\mathcal{P}$ such that $\mathcal{I}\mathcal{P}$ divides $\mathcal{I}'$. Note that $N(\mathcal{P})$ then divides $a$. Applying Theorem 2.1 in \cite{Widmer} as above to each ideal $\mathcal{I}\mathcal{Q}$ instead of $\mathcal{I}$ with $\mathcal{Q}$ a product of pairwise distinct such $\mathcal{P}$ and then using the inclusion-exclusion principle, we see that imposing (1) means that the main term gets multiplied by a factor
\[ \prod_{\mathcal{P}}{\left(1-\frac{1}{N(\mathcal{P})}\right)} \geq \prod_{p|a}{\left(1-\frac{1}{p}\right)^{[L:\mathbb{Q}]}} = \left(\frac{\phi(a)}{a}\right)^{[L:\mathbb{Q}]}\]
while the error term gets multiplied by $2^u$, where $u$ is the number of possibilities for $\mathcal{P}$. As $a \leq \mathcal{H}^d$, the factor in the main term can be bounded from below by $c_4\mathcal{H}^{-\epsilon}$ while $u$ is bounded from above by $\epsilon\log\mathcal{H}+c_5$ thanks to Th\'eor\`eme 11 in \cite{MR736719}.

We next consider (2). If (2) is not satisfied, then $\sigma(\gamma_1) = \sigma'(\gamma_1)$ for two distinct embeddings $\sigma, \sigma'$ of $L$ in $\mathbb{C}$ and so $\gamma_1$ lies in a lower-dimensional linear subspace of $\mathbb{R}^N$, obtained by equating two coordinates or setting a coordinate equal to $0$. The intersection of such a subspace with $Z_{\mathbf{Q}}$ is a bounded convex set of volume $0$ that is contained in $Z_{\mathbf{Q}}$ and so Theorem 2.1 in \cite{Widmer} shows that the number of such $\gamma_1$ can be absorbed into the error term.

It remains to be shown that the number of $\underline{\gamma}$ which satisfy conditions (1) to (4), but not (5) is of lower growth order than $\mathcal{H}^{d(l-1)-\epsilon}$. Let $\tilde{P}$ be some monic irreducible factor of $P_{\underline{\gamma}}$ in $K[t]$. Set $\tilde{l} = \deg \tilde{P}$. We define $\tilde{p}_1,\hdots,\tilde{p}_{\tilde{l}} \in K$ by $a\tilde{P}(t) = at^{\tilde{l}}+\tilde{p}_1t^{\tilde{l}-1}+\cdots+\tilde{p}_{\tilde{l}}$ and set $\tilde{K} = \mathbb{Q}(\tilde{p}_1,\hdots,\tilde{p}_{\tilde{l}})$. Since $\tilde{P}$ is irreducible in $K[t]$, $\tilde{P}$ divides $Q'_{\underline{\gamma}}\in \mathbb{Z}[t]$, and $K/\mathbb{Q}$ is Galois, we deduce that $\prod_{\sigma: \tilde{K} \hookrightarrow \mathbb{C}}{\sigma(\tilde{P})}$ is irreducible in $\mathbb{Q}[t]$ and divides $Q'_{\underline{\gamma}} \in \mathbb{Z}[t]$.

Let $\tilde{Q}$ be a minimal polynomial in $\mathbb{Z}[t]$ of some complex zero of $\tilde{P}$ and let $\tilde{a}$ denote the leading coefficient of $\tilde{Q}$, then $\tilde{Q} = \tilde{a}\prod_{\sigma: \tilde{K} \hookrightarrow \mathbb{C}}{\sigma(\tilde{P})}$. Since $\tilde{Q}$ is primitive and divides $Q'_{\underline{\gamma}}$ in $\mathbb{Q}[t]$, it follows that $\tilde{Q}$ divides $Q'_{\underline{\gamma}}$ in $\mathbb{Z}[t]$ and so $\tilde{a}$ divides $a$. Lemma \ref{lem:yetmoreusefullemmata} now implies that $a\tilde{P} \in \mathcal{O}_K[t]$, so $\tilde{p}_i \in \mathcal{O}_K$ for all $i = 1,\hdots,\tilde{l}$.

We have
\begin{equation}\label{eq:firstnormofptildei}
 |N_{K/\mathbb{Q}}(\tilde{p}_i)| = |N_{\tilde{K}/\mathbb{Q}}(\tilde{p}_i)|^{[K:\tilde{K}]} = \left(a^{[\tilde{K}:\mathbb{Q}]-1}|\tilde{q}_i|\right)^{[K:\tilde{K}]},
\end{equation}
where
\[ \tilde{q}_i = a\prod_{\sigma: \tilde{K} \hookrightarrow \mathbb{C}}{\sum_{\stackrel{\tilde{I} \subset \{\xi \in \mathbb{C};\sigma(\tilde{P})(\xi) = 0\}}{|\tilde{I}| = i}}{\prod_{\zeta \in \tilde{I}}{\zeta}}} \quad (i = 1, \hdots, \tilde{l}).\]
Since $\tilde{a}$ divides $a$, it follows from Lemma \ref{lem:yetmoreusefullemmata} that $\tilde{q}_i \in \mathbb{Z}$. Since $\tilde{Q}$ divides $Q'_{\underline{\gamma}}$, we have $|\tilde{q}_i| \leq c_6a\prod_{Q'_{\underline{\gamma}}(\zeta) = 0}{\max\{1,|\zeta|\}}$. Thanks to \eqref{eq:heightintermsofbeta}, \eqref{eq:someinside}, and \eqref{eq:othersoutside}, this implies that
\begin{equation}\label{eq:normofptildei}
|\tilde{q}_i| \leq c_6a\prod_{\stackrel{\sigma: L \hookrightarrow \mathbb{C}}{|\sigma(\beta)| \geq 1}}{|\sigma(\beta)|} = c_6\mathcal{H}^d.
\end{equation}

Since $\tilde{Q}$ divides $Q'_{\underline{\gamma}}$, we can also use \eqref{eq:heightintermsofbeta}, \eqref{eq:someinside}, and \eqref{eq:othersoutside} to estimate
\begin{multline}\label{eq:heightofptildei}
H(\tilde{p}_i) \leq a\left(\prod_{\sigma \in \Gal(K/\mathbb{Q})}{\max\left\{1,\frac{|\sigma(\tilde{p}_i)|}{a}\right\}}\right)^{\frac{1}{[K:\mathbb{Q}]}} \\
\leq c_7a\prod_{Q'_{\underline{\gamma}}(\zeta) = 0}{\max\{1,|\zeta|\}} \leq c_7a\prod_{\stackrel{\sigma: L \hookrightarrow \mathbb{C}}{|\sigma(\beta)| \geq 1}}{|\sigma(\beta)|} = c_7\mathcal{H}^d
\end{multline}
for $i = 1, \hdots, \tilde{l}$.

It then follows from \eqref{eq:firstnormofptildei}, \eqref{eq:normofptildei}, \eqref{eq:heightofptildei}, and Lemma \ref{lem:givennormboundedheight} that the number of possibilities for $\tilde{p}_i \in \mathcal{O}_K$ is bounded from above by $c_8\mathcal{H}^{d+\frac{\epsilon}{l}}$ ($i = 1,\hdots,\tilde{l}$). The leading coefficient of $a\tilde{P}$ is of course always equal to $a$. Furthermore, suppose that $P_{\underline{\gamma}} = a\tilde{P}_1\cdots\tilde{P}_m$ with $\tilde{P}_1 = \tilde{P}$ and all $\tilde{P}_i$ monic and irreducible in $K[t]$. The above argument for $\tilde{P}$ shows that $a\tilde{P}_i \in \mathcal{O}_K[t]$ for all $i$. Since $a^{m-1}P_{\underline{\gamma}} = (a\tilde{P}_1)\cdots(a\tilde{P}_m)$, we deduce that $\tilde{p}_{\tilde{l}}$ divides $a^m\beta$ in $\mathcal{O}_K$. As $m \leq l$, it follows that $N_{K/\mathbb{Q}}(\tilde{p}_{\tilde{l}})$ divides $a^{[K:\mathbb{Q}](l-1)}N_{K/\mathbb{Q}}(a\beta)$ in $\mathbb{Z}$. Lemma \ref{lem:givennormboundedheight} then shows together with \eqref{eq:heightintermsofbeta}, \eqref{eq:heightofptildei}, and elementary bounds for the divisor function that there are at most $c_9\mathcal{H}^{\frac{\epsilon}{l}}$ possibilities for $\tilde{p}_{\tilde{l}}$.

This implies that the number of possibilities for $\tilde{P}$ is bounded from above by $c_{10}\mathcal{H}^{d(\tilde{l}-1)+\frac{\epsilon\tilde{l}}{l}}$.  If $\underline{\gamma}$ satisfies conditions (1) to (4), but not (5), then $a^{-1}P_{\underline{\gamma}}$ is equal to a product of at least two such factors $\tilde{P}$. Furthermore, $\underline{\gamma}$ is uniquely determined by $P_{\underline{\gamma}}$, so it follows that the number of such $\underline{\gamma}$ is less than or equal to $c_{11}\mathcal{H}^{d(l-2)+\epsilon}$. This completes the proof of Lemma \ref{lem:hilfssatzhilfe} and thereby completes the proof of Theorem \ref{thm:conditionallowerbound}.
\end{proof}

It is now easy to show that $a(k,d)$ does not exist if $0 < k < d$ and $\gcd(k,d) > 1$. We can even determine the corresponding limes superior and limes inferior. 

\begin{thm}\label{thm:limsup}
Let $k, d \in \mathbb{N}$ such that $0 < k < d$. Then we have
\[\limsup_{\stackrel{\mathcal{H} \in B(k,d)}{\mathcal{H} \to \infty}}{\frac{\log |A(k,d,\mathcal{H})|}{\log \mathcal{H}}} = d(\gcd(k,d)-1).\]
\end{thm}

\begin{proof}
Set $l = \gcd(k,d)$. The limes superior is less than or equal to $d(l-1)$ by Theorem \ref{thm:unconditionalupperbound}(ii). If $l=1$, this already proves the theorem, so let us assume that $l \geq 2$. We want to show that the limes superior is also greater than or equal to $d(l-1)$.

We fix a totally real number field $L$ of degree $\frac{d}{l}$ that is a Galois extension of $\mathbb{Q}$. Such an $L$ can be constructed as a subfield of $\mathbb{Q}\left(\cos\left(\frac{2\pi}{p}\right)\right)$, where $p$ is prime and $p \equiv 1 \mod \frac{2d}{l}$.

Let $\sigma_1,\hdots,\sigma_{\frac{d}{l}}$ be the embeddings of $L$ into $\mathbb{R}$. The set of elements
\[(\log |\sigma_1(u)|,\hdots,\log|\sigma_{\frac{d}{l}-1}(u)|),\]
where $u$ runs over the units of $\mathcal{O}_L$, is a lattice in $\mathbb{R}^{\frac{d}{l}-1}$ by Dirichlet's unit theorem. Using elementary multidimensional diophantine approximation, we find that any lattice in $\mathbb{R}^{\frac{d}{l}-1}$ contains a vector $\left(v_1,\hdots,v_{\frac{d}{l}-1}\right)$ such that $v_i < 0$ for $i \leq \frac{k}{l}$, $v_i > 0$ for $i > \frac{k}{l}$, and $\sum_{i=1}^{\frac{d}{l}-1}{v_i} < 0$. As every algebraic unit has norm $1$, it follows that there exists an algebraic unit $u \in L$ such that $|\sigma_i(u)| < 1$ $(i \leq \frac{k}{l})$ and $|\sigma_i(u)| > 1$ ($\frac{k}{l} < i \leq \frac{d}{l}$).

We fix a prime $q$ that does not ramify in $L$. For $n \in \mathbb{N}$ sufficiently large, we can suppose that the algebraic integer $\beta = qu^n$ satisfies
\[|\sigma_i(\beta)| \leq \frac{1}{2} \quad \left(i \leq \frac{k}{l}\right)\]
and
\[|\sigma_i(\beta)| \geq 2 \quad \left(i > \frac{k}{l}\right).\]
We have $[\mathbb{Q}(\beta):\mathbb{Q}] = \frac{d}{l}$ since $\frac{k}{l}$ and $\frac{d}{l}$ are coprime, so $L = \mathbb{Q}(\beta)$.

Set $P(t) = t^l+(-1)^l\beta$. As $q$ is unramified in $L$, this polynomial is irreducible in $L[t]$ by Eisenstein's criterion for the principal ideal domain obtained by localizing $\mathcal{O}_L$ at one of the prime ideals lying over $q$. Let $\alpha$ be a complex zero of $P$. It follows that $[L(\alpha):L] = l$. Since $\beta \in \mathbb{Q}(\alpha)$ and $L = \mathbb{Q}(\beta)$, this implies that $[\mathbb{Q}(\alpha):\mathbb{Q}] = d$. Furthermore, precisely $k$ conjugates of $\alpha$ lie inside the open unit disk.

Set $\mathcal{H} = H(\alpha)$ so that $\alpha \in A(k,d,\mathcal{H})$. We have $\mathcal{H}^d = H(\beta)^{\frac{d}{l}} \in L$ as $H(\beta)^{\frac{d}{l}}$ is a product of conjugates of $\beta$ up to sign and $L/\mathbb{Q}$ is a Galois extension. Let $K$ denote the normal closure of $\mathbb{Q}(\mathcal{H}^d)$, then it follows that $K \subset L \subset \mathbb{Q}(\alpha)$. At the same time, $[\mathbb{Q}(\alpha):K] = [K(\alpha):K]$ divides $l$ by Lemma \ref{lem:degovernormalclosure}. Since $[\mathbb{Q}(\alpha):L] = l$, we deduce that $K = L$. It follows that $L = \mathbb{Q}(\alpha) \cap K$ and $N_{\mathbb{Q}(\alpha)/L}(\alpha) = \beta$, so $\alpha \in A_{L,\beta}(k,d,\mathcal{H})$ in the notation from Theorem \ref{thm:conditionallowerbound}.

We now deduce the theorem by applying Theorem \ref{thm:conditionallowerbound} with $K = L$, $\delta = \frac{1}{2}$, and $l = \gcd(k,d)$ and letting $n$ and thereby $\mathcal{H}$ go to infinity.
\end{proof}

Determining the corresponding limes inferior is even easier.

\begin{lem}\label{lem:liminf}
Let $k, d \in \mathbb{N}$ such that $0 < k < d$. Then we have
\[\liminf_{\stackrel{\mathcal{H} \in B(k,d)}{\mathcal{H} \to \infty}}{\frac{\log |A(k,d,\mathcal{H})|}{\log \mathcal{H}}} = 0.\]
\end{lem}

\begin{proof}
Thanks to Lemma \ref{lem:growthlemma} and \cite{MR1512878}, we can find $\alpha \in A(k,d)$ of arbitrarily large height such that the Galois group of the normal closure of $\mathbb{Q}(\alpha)$ is isomorphic to the full symmetric group $S_d$. The degree of any product of $k$ conjugates of such an algebraic number $\alpha$ is equal to $\binom{d}{k}$. We can therefore find arbitrarily large $\mathcal{H} = H(\alpha) \in B(k,d)$ such that $[\mathbb{Q}(\mathcal{H}^d):\mathbb{Q}] = \binom{d}{k}$. The lemma now follows from Lemma \ref{lem:usefullemma}.
\end{proof}

\section{Counting polynomials of given Mahler measure}\label{sec:mahlermeasure}
In this section, we consider polynomials of degree $d$ with integer coefficients of a given Mahler measure instead of algebraic numbers of degree $d$ of a given height. The difference is of course that we also consider reducible polynomials. For a polynomial $A \in \mathbb{Z}[t]$, we denote its Mahler measure by $M(A)$. If $\alpha$ is an algebraic number of degree $d$, its (multiplicative) height is equal to the $d$-th (positive real) root of the Mahler measure of any one of its two minimal polynomials in $\mathbb{Z}[t]$. Together with the properties that $M(a) = |a|$ ($a \in \mathbb{Z}$) and $M(AB) = M(A)M(B)$ ($A,B \in \mathbb{Z}[t]$), this characterizes the Mahler measure uniquely.

For given $d \in \mathbb{N}$, $k \in \{0,\hdots,d\}$, and $\mathcal{M} \in [1,\infty)$, we define
\begin{multline}
\widetilde{A}(k,d,\mathcal{M}) = \{ A \in \mathbb{Z}[t]; \deg A =d, M(A) = \mathcal{M} \mbox{, and precisely $k$ complex zeroes}\nonumber\\
\mbox{ of $A$ (counted with multiplicities) lie inside the open unit disk}\},\nonumber
\end{multline}
\[\widetilde{B}(k,d) = \bigcup_{\mathcal{M} \geq 1}{\{ M(A); A \in \widetilde{A}(k,d,\mathcal{M})\}},\]
and
\[\widetilde{B}(k,d,\mathcal{M}) = \widetilde{B}(k,d) \cap [1,\mathcal{M}].\]

We will prove the following analogue of Theorem \ref{thm:summary} for the Mahler measure instead of the height:

\begin{thm}\label{thm:mahlermeasure}
Let $d \in \mathbb{N}$.

If $k \in \{0,d\}$, then
\begin{equation}\label{eq:akeqd}
\lim_{\stackrel{\mathcal{M} \in \widetilde{B}(k,d)}{\mathcal{M} \to \infty}}{\frac{\log|\widetilde{A}(k,d,\mathcal{M})|}{\log \mathcal{M}}} = d
\end{equation}
and
\begin{equation}\label{eq:bkeqd}
\lim_{\mathcal{M} \to \infty}{\frac{\log|\widetilde{B}(k,d,\mathcal{M})|}{\log \mathcal{M}}} = 1.
\end{equation}

If $k \in \{1,\hdots,d-1\}$, we have
\begin{equation}\label{eq:ainf}
\liminf_{\stackrel{\mathcal{M} \in \widetilde{B}(k,d)}{\mathcal{M} \to \infty}}{\frac{\log|\widetilde{A}(k,d,\mathcal{M})|}{\log \mathcal{M}}} = 0,
\end{equation}
\begin{equation}\label{eq:asup}
\limsup_{\stackrel{\mathcal{M} \in \widetilde{B}(k,d)}{\mathcal{M} \to \infty}}{\frac{\log|\widetilde{A}(k,d,\mathcal{M})|}{\log \mathcal{M}}} = \max\{k,d-k\},
\end{equation}
and
\begin{equation}\label{eq:blim}
\lim_{\mathcal{M} \to \infty}{\frac{\log|\widetilde{B}(k,d,\mathcal{M})|}{\log \mathcal{M}}} = d+1.
\end{equation}

\end{thm}

\begin{proof}
Let $d \in \mathbb{N}$, $k \in \{0,\hdots,d\}$, $\mathcal{M} \in \widetilde{B}(k,d)$, and $\epsilon > 0$. All unspecified constants will depend only on $d$, $k$, and $\epsilon$.

We first bound $|\widetilde{A}(k,d,\mathcal{M})|$ from above: Let $A \in \widetilde{A}(k,d,\mathcal{M})$. By factoring $A$ into irreducible factors in $\mathbb{Z}[t]$, we see that $\mathcal{M} = a_0\prod_{i=1}^{s}{H(\alpha_i)^{d_i}}$ for $a_0 \in \mathbb{N}$ and algebraic numbers $\alpha_i$ of degree $[\mathbb{Q}(\alpha_i):\mathbb{Q}] = d_i$ with precisely $k_i$ conjugates inside the open unit disk. Of course, we then have $\sum_{i=1}^{s}{d_i} = d$ and $\sum_{i=1}^{s}{k_i} = k$. The number of possibilities for $s$ and the $d_i$ and $k_i$ is bounded in terms of only $d$ and $k$, so we can assume that $s$ and the $d_i$ and $k_i$ are fixed. Set $F = \mathbb{Q}(\mathcal{M})$. We claim that $H(\alpha_i)^{d_i} \in F$ for all $i=1,\hdots,s$. If not, there would exist some $\sigma \in \Gal(\bar{\mathbb{Q}}/\mathbb{Q})$ such that $\sigma(\mathcal{M}) = \mathcal{M}$, but $\sigma\left(H(\alpha_i)^{d_i}\right) \neq H(\alpha_i)^{d_i}$ for some $i$. But then it follows that
\[ \left|\sigma\left(H(\alpha_i)^{d_i}\right)\right| < |H(\alpha_i)^{d_i}|,\]
while
\[ \left|\sigma\left(H(\alpha_j)^{d_j}\right)\right| \leq |H(\alpha_j)^{d_j}|\]
for all $j \neq i$, and so $|\mathcal{M}| = |\sigma(\mathcal{M})| < |\mathcal{M}|$, a contradiction.

Now $H(\alpha_i)^{d_i} \in F$ is an algebraic integer by Lemma \ref{lem:yetmoreusefullemmata}, its height is bounded by $H(\alpha_i)^{d_i} \leq \mathcal{M}$, and $N_{F/\mathbb{Q}}\left(H(\alpha_i)^{d_i}\right)$ divides $N_{F/\mathbb{Q}}(\mathcal{M})$. Since $[F:\mathbb{Q}]$ is bounded in terms of only $d$ and $k$, we can use Lemma \ref{lem:givennormboundedheight} together with elementary bounds for the divisor function to deduce that $H(\alpha_i)^{d_i}$ is determined up to $C\mathcal{M}^{\epsilon}$ possibilities, so we can assume that $\mathcal{H}_i = H(\alpha_i)^{d_i}$ is fixed. But then $\alpha_i$ is determined up to $C_i\mathcal{H}_i^{d_i+\epsilon}$ possibilities if $k_i \in \{0,d_i\}$ and up to $\widetilde{C}_i\mathcal{H}_i^{\gcd(k_i,d_i)-1+\epsilon}$ possibilities if $0 < k_i < d_i$ by Theorems \ref{thm:heightmain}(ii) and \ref{thm:unconditionalupperbound}(ii). Note that $\prod_{i=1}^{s}{\mathcal{H}_i} \leq \mathcal{M}$.

Since $a_0$ divides $N_{F/\mathbb{Q}}(\mathcal{M})$, it is determined up to $C'\mathcal{M}^{\epsilon}$ possibilities. All in all, the number of possibilities for $A$ (given a fixed $s$ and fixed $d_i$ and $k_i$) is bounded from above by $\widetilde{C}\mathcal{M}^{(s+2)\epsilon+\max\{e,f\}}$ with 
\[ e = \max_i\{\gcd(k_i,d_i)-1;0 < k_i<d_i\} \]
and
\[ f = \max_i\{d_i; k_i \in \{0,d_i\}\}.\]
We see that $e \leq \max_i\{k_i-1\} \leq k-1 < \max\{k,d-k\}$ and $f \leq \max\{k,d-k\}$. It follows that the number of possibilities for $A$ is bounded from above by $\widetilde{C}\mathcal{M}^{\max\{k,d-k\}+(s+2)\epsilon}$. This proves that the limes superior in \eqref{eq:asup} is less than or equal to $\max\{k,d-k\}$.

For the inequality in the other direction, we consider $\mathcal{M} \in \mathbb{N}$ such that $\mathcal{M}^{\frac{1}{k}} \in B(k,k)$ (if $k \neq 0$) and $\mathcal{M}^{\frac{1}{d-k}} \in B(0,d-k)$ (if $k \neq d$). By Lemma \ref{lem:finitecomplement}, all $\mathcal{M} \in \mathbb{N}\backslash\{1\}$ satisfy these conditions. We can then apply Theorem \ref{thm:heightmain}(ii) to find many products $A(t)(t-1)^{d-k} \in \widetilde{A}(k,d,\mathcal{M})$ with $A$ equal to a minimal polynomial (in $\mathbb{Z}[t]$) of some $\alpha \in A\left(k,k,\mathcal{M}^{\frac{1}{k}}\right)$ (if $k \neq 0$) and $A(t)t^{k}\in \widetilde{A}(k,d,\mathcal{M})$ with $A$ equal to a minimal polynomial (in $\mathbb{Z}[t]$) of some $\alpha \in A\left(0,d-k,\mathcal{M}^{\frac{1}{d-k}}\right)$ (if $k \neq d$). Note that $\alpha$ is determined by $A$ up to $k$ or $d-k$ possibilities respectively. This establishes that the limes superior in \eqref{eq:asup} is greater than or equal to $\max\{k,d-k\}$. Hence, equality holds in \eqref{eq:asup}.

If $k \in \{0,d\}$ and $\mathcal{M} \in \widetilde{B}(k,d)$, then $\mathcal{M} \in \mathbb{N}$ automatically. It then follows from the above that
\[ \liminf_{\stackrel{\mathcal{M} \in \widetilde{B}(k,d)}{\mathcal{M} \to \infty}}{\frac{\log|\widetilde{A}(k,d,\mathcal{M})|}{\log \mathcal{M}}} \geq d\]
as well as
\[ \limsup_{\stackrel{\mathcal{M} \in \widetilde{B}(k,d)}{\mathcal{M} \to \infty}}{\frac{\log|\widetilde{A}(k,d,\mathcal{M})|}{\log \mathcal{M}}} \leq d. \]
We deduce \eqref{eq:akeqd}. In that case, we also have $\widetilde{B}(k,d,\mathcal{M}) = \{n \in \mathbb{N}; n \leq \mathcal{M}\}$ (for $\mathcal{M} \in [1,\infty)$) since $M\left(nt^d\right) = M\left(n(t-1)^d\right) = n$ for $n \in \mathbb{N}$, so \eqref{eq:bkeqd} follows as well.

Suppose now that $k \in \{1,\hdots,d-1\}$. We first prove \eqref{eq:ainf}. It follows from Lemma \ref{lem:growthlemma} and \cite{MR1512878} that we can find $\alpha \in A(k,d)$ of arbitrarily large height such that the Galois group of the normal closure of $\mathbb{Q}(\alpha)$ is isomorphic to the full symmetric group $S_d$. Any product of $k$ conjugates of such an algebraic number $\alpha$ has degree $\binom{d}{k}$. We can therefore find arbitrarily large $\mathcal{M} = H(\alpha)^d \in \widetilde{B}(k,d)$ such that $[\mathbb{Q}(\mathcal{M}):\mathbb{Q}] = \binom{d}{k}$.

Let now $\mathcal{M} \in \widetilde{B}(k,d)$ be arbitrary with $[\mathbb{Q}(\mathcal{M}):\mathbb{Q}] = \binom{d}{k}$ and let $A \in \widetilde{A}(k,d,\mathcal{M})$. Suppose that $A$ decomposes in $\mathbb{Z}[t]$ as the product of a non-zero integer $a_0$ and $s$ irreducible factors $A_i$ of degree $d_i$ and with $k_i$ complex zeroes inside the open unit disk respectively ($i=1,\hdots,s$). We can then bound the degree of $\mathcal{M}$ from above by $\prod_{i=1}^{s}{\binom{d_i}{k_i}}$. Using the combinatorial interpretation of the binomial coefficient, one can see that $\binom{d'}{k'}\binom{d''}{k''} < \binom{d'+d''}{k'+k''}$ if $d', d'' \in \mathbb{N}$, $k' \in \{0,\hdots,d'\}$, $k'' \in \{0,\hdots,d''\}$, and $(k',k'') \not\in \{(0,0),(d',d'')\}$. If $s > 1$, this implies that $\prod_{i=1}^{s}{\binom{d_i}{k_i}}$ is smaller than $\binom{d}{k}$, and we obtain a contradiction.

We deduce that $s = 1$. Therefore, $A$ must be equal to the product of a non-zero integer $a_0$ and a minimal polynomial (in $\mathbb{Z}[t]$) of some algebraic number $\alpha$ of degree $d$. We want to bound the number of possibilities for $a_0$ and $\alpha$.

Since $H(\alpha)^d$ is an algebraic integer, $a_0$ divides $N_{\mathbb{Q}(\mathcal{M})/\mathbb{Q}}(\mathcal{M})$ and so the number of possibilities for $a_0$ is bounded by $C''\mathcal{M}^{\epsilon}$. Since $[\mathbb{Q}\left(H(\alpha)^d\right):\mathbb{Q}] = [\mathbb{Q}(\mathcal{M}):\mathbb{Q}] = \binom{d}{k}$, the number of possibilities for $\alpha$, given $a_0$, is bounded by $C'''H(\alpha)^{\epsilon} \leq C'''\mathcal{M}^{\epsilon}$ thanks to Lemma \ref{lem:usefullemma}. We deduce \eqref{eq:ainf}.

We can deduce from Theorem \ref{thm:heightmaintoo}(ii) that the limit in \eqref{eq:blim} has to be greater than or equal to $d+1$ (if it exists). For the inequality in the other direction (which will also imply the existence of the limit), we can use that for $\mathcal{M} \in [1,\infty)$, any $\widetilde{\mathcal{M}} \in \widetilde{B}(k,d,\mathcal{M})$ is equal to $M(A)$ for some $A \in \mathbb{Z}[t]$ of degree $d$, and that the $d+1$ coefficients of this $A$ are all bounded by $2^dM(A) \leq 2^d\mathcal{M}$ in absolute value thanks to Lemma 1.6.7 in \cite{MR2216774}.
\end{proof}

\section{Dynamics of the height function}\label{sec:heightdynamic}
In this section, we study the dynamics of the restriction of the height function to $\bar{\mathbb{Q}} \cap \mathbb{R}$.

We start by classifying the periodic points. We define inductively $H^0 = \id$ and $H^n = H \circ H^{n-1}$ ($n \in \mathbb{N}$).

\begin{thm}\label{thm:periodic}
If $n \in \mathbb{N}$ and $\alpha \in \bar{\mathbb{Q}}$ are such that $H^n(\alpha) = \alpha$, then $\alpha = a^b$ for some $a \in \mathbb{N}$ and $b \in \mathbb{Q}$, $b > 0$, and $H(\alpha) = \alpha$. Conversely, $H(a^b) = a^b$ for all $a \in \mathbb{N}$ and $b \in \mathbb{Q}$, $b > 0$.
\end{thm}

The proof of this theorem will be essentially achieved by the following lemma:

\begin{lem}\label{lem:lemma}
If $\alpha \in \bar{\mathbb{Q}}$, then $H(\alpha) \geq H(H(\alpha))$ with equality if and only if $H(\alpha) = a^b$ for some $a \in \mathbb{N}$ and $b \in \mathbb{Q}$, $b > 0$.
\end{lem}

\begin{proof}
Let $\beta = H(\alpha)$. We have $\beta^{[\mathbb{Q}(\alpha):\mathbb{Q}]} = a'\prod'_{|\gamma| \geq 1}{\gamma}$, where $a'$ is the leading coefficient of a minimal polynomial of $\alpha$ in $\mathbb{Z}[t]$ and the product runs over all complex zeroes $\gamma$ of that minimal polynomial that are at least $1$ in absolute value. It is now clear that any conjugate of $\beta^{[\mathbb{Q}(\alpha):\mathbb{Q}]}$ that is not equal to $\beta^{[\mathbb{Q}(\alpha):\mathbb{Q}]}$ is less in absolute value than $\beta^{[\mathbb{Q}(\alpha):\mathbb{Q}]}$. Furthermore, $\beta^{[\mathbb{Q}(\alpha):\mathbb{Q}]}$ is an algebraic integer by Lemma \ref{lem:yetmoreusefullemmata}.

It follows that $H(\beta)^{[\mathbb{Q}(\alpha):\mathbb{Q}]} = H\left(\beta^{[\mathbb{Q}(\alpha):\mathbb{Q}]}\right) < \beta^{[\mathbb{Q}(\alpha):\mathbb{Q}]}$ and hence $H(H(\alpha)) < H(\alpha)$ unless $[\mathbb{Q}\left(\beta^{[\mathbb{Q}(\alpha):\mathbb{Q}]}\right):\mathbb{Q}] = 1$, in which case $\beta = a^b$ for some $a \in \mathbb{N}$ and $b \in \mathbb{Q}$, $b > 0$. Furthermore, it is clear that $H(a^b) = a^b$ for all $a \in \mathbb{N}$ and $b \in \mathbb{Q}$, $b > 0$.  
\end{proof}

We can now prove Theorem \ref{thm:periodic}.

\begin{proof}[Proof of Theorem \ref{thm:periodic}]
Suppose that $H^n(\alpha) = \alpha$ for some $n \in \mathbb{N}$ and $\alpha \in \bar{\mathbb{Q}}$. It follows from Lemma \ref{lem:lemma} that
\[ H^n(\alpha) = H^n(H^n(\alpha)) \leq H(H(H^{n-1}(\alpha))) \leq H(H^{n-1}(\alpha)) = H^n(\alpha).\]
We deduce that equality must hold everywhere. Hence, Lemma \ref{lem:lemma} implies that $\alpha = H^n(\alpha)$ is of the desired form and we have $H(\alpha) = \alpha$. The converse implication is again obvious.
\end{proof}

Next, we study the possibilities for the forward orbit of a given element.

\begin{thm}\label{thm:attractors}
Let $\alpha \in \bar{\mathbb{Q}}$ and define inductively $\alpha_0 = \alpha$, $\alpha_n = H(\alpha_{n-1})$ ($n \in \mathbb{N}$). Then either there exist $N, a \in \mathbb{N}$ and $b \in \mathbb{Q}$, $b > 0$, such that $\alpha_n = a^b$ for all $n \geq N$ or $\lim_{n \to \infty}{\alpha_n} = 1$.
\end{thm}

Theorem \ref{thm:attractors} implies (the non-trivial direction of) Theorem \ref{thm:periodic} as a corollary.

\begin{proof}
Let $d = [\mathbb{Q}(\alpha_1):\mathbb{Q}]$. We will show by induction on $n \in \mathbb{N}$: Either $\alpha_m = a^b$ for some $m, a \in \mathbb{N}$, $m \leq n$, and $b \in \mathbb{Q}$, $b > 0$ (and then of course the same holds for all $\alpha_r$ with $r \geq m$) or $\alpha_n^{d!^{n-1}}$ is the product of at most $(d!-1)^{n-1}$ conjugates of $\alpha_1$ up to sign, where $0^0 := 1$ and the empty product is defined to be $1$.

The assertion is trivially true for $n = 1$. Suppose now that it has been proven for all $m \leq n$ ($n \in \mathbb{N}$) and suppose further that no $\alpha_m$ is of the form $a^b$ for some $m, a \in \mathbb{N}$, $m \leq n$, and $b \in \mathbb{Q}$, $b > 0$. It follows that $\alpha_n^{d!^{n-1}}$ is a product of at most $(d!-1)^{n-1}$ conjugates of $\alpha_1$ up to sign. Since $\alpha_1$ is an algebraic integer by Lemma \ref{lem:yetmoreusefullemmata}, so is $\alpha_n^{d!^{n-1}}$. Now $\alpha_n^{d!^{n-1}}$ lies in the normal closure of $\mathbb{Q}(\alpha_1)$ and so its degree divides $d!$. It follows that $\alpha_{n+1}^{d!^n} = H\left(\alpha_n^{d!^{n-1}}\right)^{d!}$ is equal to the product of the absolute values of at most $d!$ conjugates of $\alpha_n^{d!^{n-1}}$, namely of absolute values of conjugates of $\alpha_n^{d!^{n-1}}$ outside the open unit disk, the absolute value of each such conjugate occurring precisely $d![\mathbb{Q}(\alpha_n^{d!^{n-1}}):\mathbb{Q}]^{-1}$ times in the product. Since the non-real conjugates appear in complex conjugate pairs in the product and the real conjugates are equal to their absolute values up to sign, we deduce that $\alpha_{n+1}^{d!^n}$ is equal to the product of at most $d!$ conjugates of $\alpha_n^{d!^{n-1}}$ up to sign, each conjugate occurring either $0$ or $d![\mathbb{Q}(\alpha_n^{d!^{n-1}}):\mathbb{Q}]^{-1}$ times in the product.

But if each conjugate of $\alpha_n^{d!^{n-1}}$ occurs $d![\mathbb{Q}(\alpha_n^{d!^{n-1}}):\mathbb{Q}]^{-1}$ times in this product, then $\alpha_{n+1}^{d!^n}$ must be a rational integer and so $\alpha_{n+1}$ is of the form $a^b$ for some $a \in \mathbb{N}$ and $b \in \mathbb{Q}$, $b > 0$. Otherwise, $\alpha_{n+1}^{d!^n}$ is equal to a product of at most $d!-1$ conjugates of $\alpha_n^{d!^{n-1}}$ up to sign and therefore equal to a product of at most $(d!-1)^n$ conjugates of $\alpha_1$ up to sign.

If no $\alpha_n$ is of the form $a^b$ for some $a \in \mathbb{N}$ and $b \in \mathbb{Q}$, $b > 0$, then it follows directly that
\[ 1 \leq \alpha_n = |\alpha_n| \leq \alpha_1^{\left(1-\frac{1}{d!}\right)^{n-1}}\]
for all $n \in \mathbb{N}$ since every conjugate of $\alpha_1$ is less than or equal to $\alpha_1$ in absolute value. We deduce that $\lim_{n \to \infty}{\alpha_n} = 1$.
\end{proof}

\section*{Acknowledgements}
This article has grown out of an appendix to my PhD thesis. I thank my PhD advisor Philipp Habegger for his constant support and for many helpful and interesting discussions. I thank Philipp Habegger and Ga\"el R\'emond for helpful comments on the thesis and I thank Ga\"el R\'emond for suggesting the proof of Lemma \ref{lem:finitecomplement}. I thank Fabrizio Barroero for useful comments on an earlier version of this article and I thank Martin Widmer for correspondence on his work. I thank the referee for their helpful suggestions for improving the exposition. When I had the initial idea for this article, I was supported by the Swiss National Science Foundation as part of the project ``Diophantine Problems, o-Minimality, and Heights", no. 200021\_165525. I completed it while supported by the Early Postdoc.Mobility grant no. P2BSP2\_195703 of the Swiss National Science Foundation. I thank the Mathematical Institute of the University of Oxford and my host there, Jonathan Pila, for hosting me as a visitor for the duration of this grant.

\bibliographystyle{acm}
\bibliography{Bibliography}
\end{document}